\documentclass[oneside,leqno]{amsart}
\usepackage{latexsym,amssymb,amsmath,graphics,color}
\usepackage{tabularx}

\def\B{{\mathcal{B}}}

\def\E{{\mathbb{E}}}
\def\H{{\mathcal{H}}}

\def\N{{\mathbb{N}}}
\def\P{{\mathbb{P}}}
\def\Q{{\mathbb{Q}}}
\def\R{{\mathbb{R}}}

\newcommand{\F}{{\mathcal F}}

\newcommand{\8}{\infty}
\renewcommand{\d}{\delta}
\renewcommand{\a}{\alpha}
\renewcommand{\b}{\beta}
\newcommand{\D}{\Delta}
\renewcommand{\O}{\Omega}
\renewcommand{\l}{\lambda}

\newcommand{\eps}{\varepsilon}
\newcommand{\z}{\zeta}

\newcommand{\s}{\sigma}

\newcommand{\ov}{\overline}
\newcommand{\wt}{\widetilde}

\newcommand{\1}{{\bf 1}}
\newcommand{\wh}{\widehat}

\renewcommand{\o}{\omega}

\newcommand{\supp}{\mathrm{supp}}

\newcommand{\is}[2]{\langle #1,#2\rangle}
\newcommand{\bis}[2]{\bigg\langle #1,#2\bigg\rangle}

\newcommand{\en}{ \E\big[ N \big]}

\newcommand{\Ps}{P^s}
\newcommand{\Qs}{Q^s}
\newcommand{\Qxs}{\Q_x^s}
\newcommand{\Pst}{P_*^s}
\newcommand{\es}{e^s}
\newcommand{\est}{e_*^s}
\newcommand{\nus}{\nu^s}
\newcommand{\nust}{\nu_*^s}
\newcommand{\skalar}[1]{\langle #1 \rangle}
\newcommand{\abs}[1]{\left| #1 \right|}

\newcommand{\norm}[1]{\left\| #1 \right\|}
\newcommand{\Erw}[2][]{\E_{#1} \left[ #2 \right]}

\newcommand{\red}{\color{black}}

\newtheorem{thm}[equation]{Theorem}
\newtheorem{cor}[equation]{Corollary}
\newtheorem{lem}[equation]{Lemma}
\newtheorem{lemma}[equation]{Lemma}

\newtheorem{prop}[equation]{Proposition}
\theoremstyle{definition}

\newtheorem{rem}[equation]{Remark}
\newtheorem{ex}[equation]{Example}

\numberwithin{equation}{section}

\begin{document}

\title{On  multidimensional  Mandelbrot cascades}
%\dedicatory{\normalsize Dedicated to Jean-Pierre Kahane, with admiration}
\author[D. Buraczewski, E. Damek, Y. Guivarc'h, S. Mentemeier]
{Dariusz Buraczewski, Ewa Damek, Yves Guivarc'h, Sebastian Mentemeier}
\address{D. Buraczewski, E. Damek, S. Mentemeier\\
Uniwersytet Wroclawski\\ Instytut Matematyczny\\
 50-384 Wroclaw\\pl. Grunwaldzki 2/4\\
 Po\-land}
\email{\{dbura, edamek, mente\}@math.uni.wroc.pl}
\address{Y. Guivarc'h\\ IRMAR, Universit\'e de Rennes 1\\
Campus de Beaulieu\\
35042 Rennes cedex, France} \email{yves.guivarch@univ-rennes1.fr}

\thanks{
D.~Buraczewski and E.~Damek were partially supported by NCN
grant DEC-2012/05/B/ST1/00692.}

\begin{abstract}
Let $Z$ be a random variable with values in a proper closed convex
cone $C\subset \R^d$, $A$ a random endomorphism of $C$ and $N$ a random integer. We assume that $Z$, $A$, $N$ are independent. Given $N$ independent copies $(A_i,Z_i)$ of $(A,Z)$ we define a new random variable $\wh Z = \sum_{i=1}^N A_i
Z_i$. Let $T$ be the corresponding transformation on the set of
probability measures on $C$ i.e. $T$ maps the law of $Z$ to the law of $\wh Z$. If the matrix $\E[N] \E [A]$ has dominant eigenvalue 1, we study existence and properties of fixed points of $T$ having finite nonzero expectation. Existing one dimensional results concerning T are
extended to higher dimensions. In particular we give conditions under which such fixed points of $T$ have multidimensional regular variation in the sense of extreme value theory 
and we  determine the index of regular variation.
\end{abstract}

%\subjclass[2000]{.}

\maketitle

\section{Introduction}
\label{sec: introduction}

\subsection{The smoothing transform}
We consider the vector space $V=\R^d$ endowed with a scalar
product $\is xy$ and the corresponding norm $x\to|x|$. We equip the space of endomorphisms of $V$, ${\rm End}(V)$ with the associated operator norm $\norm{a}:= \sup_{\abs{x}=1} \abs{ax}$. %For a measurable space %$E$ let $M^1(E)$ be the space of probability measures defined on $E$.
Let $\mu$ be a probability measure on ${\rm End}(V)$, i.e. $\mu \in M^1({\rm End}(V))$. Suppose that $A$ is a random endomorphism distributed according to $\mu$. Let $N$ be
a random integer and $Z$ a $V$-valued random vector
such that $A$, $N$ and $Z$ are independent.
We consider $N$ independent copies $(A_i,Z_i)$ ($1\le i\le N$) of $(A,Z)$
and the new random variable $\wh Z$ defined by
\begin{equation}
\label{iteration}
\wh Z= \sum_{i=1}^N A_i Z_i.
\end{equation}
Thus we obtain a transformation of $M^1(V)$ defined by $\rho\to T\rho$
where $\rho$ is the law of $Z$ and $T\rho$ the law of $\wh Z$.
%We observe that $T$ commutes with the action of $\R_+^*$ on measures defined by
%extension of the dilations $x\to tx$ ($t\in \R_+^*$).

Nontrivial fixed points of $T$ ($\rho\not=\d_0$) and their tails have been of considerable interest. As we shall see below, under natural
conditions, there exists a non trivial fixed point. Heuristically, this
corresponds to the competing effects of expansion by summation ($N>1$) and contraction
by endomorphism $A_i$. Furthermore, if $A$ is also expanding with positive probability, there exists $\chi > 1$, such that for $s \ge \chi$, the $s$-moment of the fixed point (which will be proven to be essentially unique) of equation \eqref{iteration} is infinite, in particular, it has heavy tails.

For $d=1$ and $A_i>0$, fixed points of $T$ were considered by  Durrett and Liggett \cite{DL},  Holley \cite{Ho}, Spitzer \cite{S}, who  studied invariant measures of infinite systems of particles in interaction.
Independently, in the context of random fractals, various questions on equation \eqref{iteration} were considered by Mandelbrot \cite{M} and some of them were solved by Kahane and Peyri\`ere \cite{KP}.
%They studied their invariant measures. Equation \eqref{iteration} is a limiting case of more sophisticated equations in \cite{Ho,S}. It corresponds to the case %when the particles are randomly moving on a tree instead of the $n$-dimensional lattice, and the random interaction is given by $A_i$.
For the most general contributions see Alsmeyer, Biggins, Meiners \cite{ABM2012}, Biggins, Kyprianou \cite{BK} and Liu \cite{L}.
If $N$ is constant and $N\E A=1$, solutions of the fixed point
equation with finite mean play an important role in the context
of construction of a large class of self-similar random measures
\cite{GMW,L0}.  Heavy tail properties of the the fixed points were studied  by Guivarc'h \cite{G}, Jelenkovic, Olvera-Cravioto \cite{JO}, Liu \cite{L} and R\"osler, Topchii, Vatutin \cite{RTV}. Finally, equation \eqref{iteration} appeared also in the context of branching random walks; see Biggins \cite{Bi}.

In the one-dimensional situation existence of solutions was  discussed by Kahane, Peyri\`ere
\cite{KP} in the case when $N$ is a constant and $\E A =1/N$. For very general results, in particular also concerned with uniqueness, see \cite{ABM2012, AMeiners, DL, L}.
The behavior of the tails of fixed points depends on the properties
of the function $\theta(s) = \E\big[ \sum_{i=1}^N
A_i^{s}\big]$ (here $A_i$ and $N$ can be dependent), see
\cite{DL}. Then if $\theta(1)=1$, $\theta'(1)<0$ and $\theta(\chi)=1$ for some
$\chi>1$, Guivarc'h \cite{G} and Liu \cite{L} proved that if ${\rm
supp\ \mu}$ is non arithmetic, then the fixed points have heavy tails, i.e. if the law of $Y$ is a fixed point of $T$, then $\lim_{t\to\8} t^{\chi}\P[Y>t]$ exists and  is
positive. Recently, asymptotic properties of solutions of
\eqref{iteration} in the boundary case, when $\theta'(1)=0$ were also
described (see \cite{BK,Bu}).

The multidimensional case ($d>1$) was studied recently in \cite{BDMM, Mentemeier, Mirek}. In \cite{BDMM} the authors consider two classes of invertible matrices: similarities  (products of dilations and orthogonal matrices) and general matrices, however under quite restrictive assumptions (continuity of the distribution and irreducibility of the action on the unit sphere, see \cite{AM, BDMM} for more details). Fixed points of $T$ in the multidimensional situation describe e.g. equilibrium distributions of kinetic gas models in statistical physics (see \cite{BL, BM, BDMM}, here $d=3$), or the joint asymptotics of the number of key comparisons and key exchanges in Quicksort (see \cite{NR2006}, in fact, there an inhomogeneous version of equation \ref{iteration} is considered). The multidimensional equation can be also interpreted in the context of 'colored' particles numbered from 1 to $d$ randomly moving on a tree, \cite{BN}.

\medskip

In this paper we consider the multidimensional situation ($d>1$) under assumption that the support of $\mu $ leaves invariant a proper closed convex cone $C \subset \R^d$, e.g. $\R^d_+=[0,\8)^d$. We study existence of fixed points and properties of their tails, and we prove analogues of the results of \cite{KP, G}. Our setting includes nonnegative matrices as considered by Kesten \cite{K1} (we will strengthen several of his results about the action of products of such matrices) and also some other classes of matrices being  natural generalizations of such. In particular, in contrast to \cite{BDMM, Mirek}, we do not assume the matrices to be invertible.  Below, after giving an ad-hoc version of our main result, we will describe an example where the multidimensional equation with nonnegative and noninvertible matrices is explicitly needed. Precise statements of the main results will be given in the subsequent section, after introducing some more concepts. For a preliminary version of this paper see \cite{BDG0}.

We thank the referees for useful remarks which helped to improve strongly the previous version of the paper. 

\subsection{Ad-hoc version of the main result}

At first, we need a few pieces of notation, namely a multidimensional analogue of the function $\theta(s)$. Let $(A_i)_{i \in \N}$ be a sequence of i.i.d. copies of $A$ (independent of $N$) and introduce
$$ \kappa(s) := \lim_{n \to \infty} \Erw{\norm{A_n \dots A_1}^s}^{1/n}.$$ Then $\Erw{N} \kappa(s)$ will play the role of $\theta(s)$. It will be shown that this function is log-convex and that $\kappa(1)$ equals the spectral radius of $\Erw{A}$. Write $\delta_v$ for the Dirac measure in $v$ and $\lambda_a$ for the Perron-Frobenius eigenvalue of a positive matrix (i.e all entries $>0$).

\medskip

There is an obvious way to construct a fixed point of $T$, namely iteration. This can be done as follows:
Set
\begin{equation}\label{defAi}A^i = \1_{\{ i \le N\}}A_i,
\end{equation} i.e. the matrices with indexes larger than the random number $N$ are just zero, while the others are i.i.d.
 Consider the Ulam-Harris tree $J=\bigcup_{n=0}^\infty \N^n$ with root $\emptyset$. For a vertex $\gamma \in \N^n$ write $\abs{\gamma}=n$ for its generation. Assign to every vertex $\gamma$ an i.i.d. copy $(A^i(\gamma))_{i \in \N}$ of $(A^i)_{i \in \N}$ and let $\F_n$ be the $\s$-field
$$ \F_n := \sigma\Bigl((A^i(\gamma))_{i \in \N})\ : \ \abs{\gamma} < n   \Bigr).$$
  One should think of the matrix $A^i(\gamma)$ as a weight along the edge connecting $\gamma$ with its $i$-th successor $\gamma i$. The product of weights along the shortest path connecting $\gamma$ to the root is then defined recursively by
 $$L(\emptyset)= {\rm Id} \qquad \text{and} \qquad L(\gamma i) = L(\gamma) A^i(\gamma).$$
 Given a nonzero vector $v \in C$ with $\E N\, \E A\, v = v$,
 we consider the sequence of random
variables
\begin{equation}\label{mandelbrot cascade}Y_n := \sum_{|\gamma|=n}L({\gamma})v
\end{equation} called Mandelbrot's
cascade (or weighted branching process associated with $(A^i)_{i \in \N}$ and $v$), a main feature being that the law of $Y_n$ equals $T^n \delta_v$, while the process $Y_n$ forms a martingale w.r.t. $\F_n$, which will be shown to converge a.s. to a random variable $Y$.

\begin{thm}\label{thm:adhoc}
Let $\mu$ be a probability measure on the set  of nonnegative random matrices with no zero row and no zero column. Assume that $\supp\, \mu$ contains a positive matrix  and that there is a positive vector $w$ such that $\Erw{N}\, \Erw{A}\, v = v$ (the existence of such $w$ is equivalent to $\Erw{N}\, \kappa(1)=1$). Let $N \ge 2$ a.s. and $\Erw{N^2} < \infty$. Then
$$ \text{ $Y$ is a non-trivial fixed point of $T$   $\quad \Leftrightarrow \quad$ the (left) derivative $\kappa'(1^-) < 0$.} $$

\medskip

\noindent Assume in addition, that $N \ge 2$ is constant and that the multiplicative subgroup generated by
$ \{ \lambda_a \, : \, a \in \supp\, \mu, \text{ $a$ is positive }\}$
is dense  in $\R_+$. Assume that $\P(\is Y u = r)=0$ for all nonnegative $r$ and $u \in (0,\infty)^d$. If there is $\chi >1$ s.t. $N \kappa(\chi)=1$ and some moment conditions on $A$ are satisfied, then there is a continuous function $D(x)=\abs{x}^\chi D(\frac{x}{\abs{x}})$ on $(0,\8)^d$, such that
$$ \lim_{t \to \infty} t^\chi\P(\skalar{Y,x}>t ) = D(x) $$
for all $x \in (0,\8)^d$. The function $D(x)$ is strictly positive.
\end{thm}

We close the introduction by giving the afore-mentioned example.

\begin{ex}
Menshikov, Petritis and Popov \cite{MPP} study positive recurrence of the so-called bindweed model and obtain a necessary and sufficient condition, namely the a.s convergence of the series
\begin{equation}\label{MPPsum} \sum_{n =0}^\8 \, \sum_{\abs{\gamma}=n} L(\gamma) \end{equation}
(in our notation). In their case, the entries of the matrix $A$ are ratios of random transition probabilities, hence nonnegative. Moreover they assume that $A$ is positive a.s., but not necessarily invertible. Thus the assumptions of the first part of Theorem \ref{thm:adhoc} are satisfied.

In \cite[Theorem 2.5]{MPP} they show that
\begin{itemize}
\item the series converges, if $\inf_{s \in (0,1]} \Erw{N} \kappa (s) < 1$ whereas
\item the series diverges, if $\inf_{s \in (0,1]} \Erw{N} \kappa(s) > 1$.
\end{itemize}
It remained an open question, what happens if $\inf_{s \in (0,1]} \Erw{N} \kappa(s) = 1$. But this corresponds exactly to the case $\Erw{N}\, \kappa(1)=1$ with $\kappa'(1^-) \le 0$.
Our result shows that if $\kappa'(1^-) < 0$, then $Y_n = \sum_{\abs{\gamma}=n} L(\gamma) v$ converges a.s. to a non degenerate limit, hence the sum in \eqref{MPPsum}
is divergent.

\end{ex}

\section{Statement of main results and notation}

In this section, we will first introduce the precise hypothesis which we are going to impose on the law $\mu$ of the random matrices resp. its support $\supp\ \mu$. Then we will state our main theorems and conclude with an outline of the paper.

\subsection{Notation and hypotheses}

When considering a multivariate problem, one usually has to impose moment conditions which are similar to those known from the one-dimensional situation and additionally some purely multivariate, mainly geometric assumptions. This applies here as well, and we will start with the geometric part, i.e. assumptions on the support of $\mu$.

\subsubsection{Geometric assumptions}
We fix a
proper closed convex subcone $C\subset V$ with nonempty interior
$C^0\not= \emptyset$. Proper means that the cone must be contained in a halfspace of $V$.
Let $$C^* = \{x\in V; \is
xy\ge 0 \mbox{ for any } y\in C\}$$ be the dual cone of $C$.  Then $C^*$ is necessarily closed and convex with non-empty interior. % See LN, middle of p. 4
We  adopt for $C^*$ similar
notations as for $C$ and write $a^*$ for the dual map of $a$
defined by $\is{a^*x}y = \is x{ay}$ ($x,y\in V$).

 Introducing $C_+= C \setminus\{0\}$, set
 % and $C_+^* = C^* \setminus \{0\}$. Let $S$ be the semigroup of elements of ${\rm End}(V)$ that leave $C_+$ and $C^*_+$ invariant, i.e.
\begin{equation}\label{def:S}
S~:=~ \{ a \in {\rm End}(V) \ : \ a C_+ \subset C_+, \ a^* C^*_+ \subset C^*_+\}.
\end{equation}

Defining  $$C_1 ~:=~ \{x\in C; |x|=1\}$$ as the intersection of $C$ with the unit sphere, we see that for all $a \in S$ its action on $C_1$ is well defined by
$$ a \cdot x := \frac{ax}{\abs{ax}}$$ and we write
$$\iota(a):=\inf_{x\in C_1}|ax|.$$ Since $C_1$ is compact
we
have that $\iota(a)>0$ for $a\in S$.
%Let $S$
%be the subsemigroup of the linear group $G = {\rm GL} (V)$ consisting of elements preserving the cone $C$, i.e. $a\in S$ if $a(C)\subset C$.

% and $S_+$ the open subsemigroup of $S$ defined by $S_+ = \{a\in S;{\rm Ker}\ a\cap C = \{0\}\}$.
Let
\begin{equation}
S^0 ~=~ \{a\in S\ :\ aC_+\subset C^0\} ~=~ \{ a \in S\ :  a^* C_+^* \subset (C^*)^0 \}.
\end{equation}
According to the Krein-Rutman theorem, each element $a$ of $S^0$ has a positive dominant eigenvalue
$\lambda_a$ with corresponding dominant eigenvector $v_a\in C^0$. We will assume $v_a$ to be normalized, i.e. $v_a \in C_1$. %with $d> 1$, $S$ as defined above,
%Furthermore $a\in S_0$ acts on $C_1=\{x\in C; |x|=1\}$ by $x\to a\cdot x = \frac{ax}{|ax|}$.
%We observe that $S$ is a
%proper closed convex subcone of ${\rm End }\ V$ and, as it is easily
%shown,  $S^0\subset S_+$ is the interior of $S$.

%Let $C^* = \{x\in V; \is xy\ge 0 \mbox{ for any } y\in C\}$.

%Hence $a\in \Gamma$ acts on the compact sets $C_1$ and $C_1^*$ as continuous projective transformations.

 If the cone $C$ is $(\R_+)^{d}$, for $\R_+ = [0,\8)$, then $C^*=C$, $S$ is the semigroup of matrices with nonnegative entries that have neither a zero row nor  a zero column
and  $S^0$ consists of matrices with strictly positive entries. Another special
case we have in mind is the cone $C$ of real positive semi-definite
matrices. Then $S$ contains the set of mappings $M\to aMa^t$,
where $a$ is an invertible
matrix.
%\footnote{But the set $S^0$ is empty!!! One might consider sets of positive semi-definite matrices with common eigenspaces; such that in addition $a M a^t \in C^0$ for any $a,M$... OR SUBSETS OF THAT CONE, WITH QUOTIENTS OF EIGENVALUES FROM A FIXED INTERVAL...}}
More generally, $C$ can be a symmetric cone (e.g. the light
cone, see \cite{FK}) or a homogeneous cone \cite{V}.

%, and $\mu$ is a probability measure on $S$.
We assume that $A$ is a random element of $S$ distributed according to a given probability measure $\mu\in M^1(S)$.
We denote by $[{\rm supp}\ \mu]$ the smallest closed
subsemigroup of $S$ containing ${\rm supp}\ \mu$ and we will assume that it satisfies the following condition $(\H)$:

\medskip

We say that a subsemigroup $\Gamma$ of $S$ satisfies condition ($\H$) (compare \cite{H}), if
\begin{itemize}
\item[(a)] each $a \in \Gamma$ is \emph{allowable}, i.e. $aC^0 \subset C^0$ and $a^*\, (C^*)^0 \subset (C^*)^0$, and \item[(b)] $\Gamma \cap S^0 \neq \emptyset$.
\end{itemize}
%\begin{itemize}
%\item[a)] no  proper subspace $W\subset V$ with $W\cap C\not=\{0\}$ satisfies $\Gamma W\subset W$;
%\item[b)] $\Gamma\cap S^0 \not= \emptyset$;
%\item[c)] no proper subspace $W'\subset W$ with $W'\cap C^*\not=\{0\}$ satisfies $\Gamma^* W' \subset W'$.
%\end{itemize}
%{\red Condition $\H$ is an analogue, in the cone context, of the so called
%i-p condition in \cite{GL, GLnotes}.}
%Observe that $\Gamma$ satisfies $(\H)$ if and only if $\Gamma^*$ does.

\medskip
\begin{rem}
\begin{enumerate}
\item
The definition of $S$ as well as condition $(\H)$ are such that the roles of $C$ and $C^*$ can be interchanged, i.e. all results that will be proven for matrices acting on $C$ under hypotheses hold as well for their adjoint matrices acting on $C^*$ and vice versa.
\item Observe that $\Gamma=\{ a^n \, ; \, n \in \N \}$ for some $a \in S^0$ would be a legal choice, i.e. $\Gamma$ might be quite degenerate. This is the main reason why we sometimes have to impose a regularity condition on the fixed points, namely that  $\P(\is Y u = r)=0$ for all nonnegative $r$ and $u \in C^*$.
\item  It is enlightening to compare $(\H)$ with the so-called i-p condition (i-p for irreducibility and proximality) which has been studied intensively in \cite{Fu,GL,GLnotes}: A closed subsemigroup $G$ of the group of invertible matrices ${\rm GL}(V)$ satisfies the i-p condition, if
\begin{itemize}
\item (irreducible): $G$ is {\em strongly irreducible}, i.e. no finite union $\bigcup_{i=1}^n W_i$ of proper subspaces $W_i \subsetneq V$ satisfies
\begin{equation}\label{strongirred} G \left( \bigcup_{i=1}^n W_i \right) \subset \bigcup_{i=1}^n W_i, \end{equation}
\item (proximal): $G$ contains at least one element with a unique simple dominant eigenvalue.
\end{itemize}
The proximality assumption corresponds to part (b) of $(\H)$, while  $(a)$ is always satisfied  for $a \in {\rm{GL}(V)}$ that leave $C$ and $C^*$ invariant, since it maps open sets into open sets. { In the present work, we do not assume irreducibility in general, but we will use a weaker version relative to $C$ which is a sufficient condition e.g. for the regularity of fixed points, in the sense mentioned above. }

%If $G \subset S$, then strong irreducibility is already implied by the assumption that there is no proper $G$-invariant subspace, see Lemma \ref{lem:strongirreducible}.

\end{enumerate}
\end{rem}

We will also be led to consider the following aperiodicity condition: We say that $\Gamma$ is aperiodic, if
\begin{equation} \tag{A}
\Delta(\Gamma):= \{\lambda_a:a\in\Gamma\cap S^0\} \text{ generates a dense multiplicative subgroup of } \R_+.
\end{equation}
We observe that if $\Gamma\subset {\rm GL}(V)$, ${\rm dim } V>1$ and $\Gamma$ satisfies condition $i-p$, then $\Gamma$ is aperiodic, see \cite{GU}.
{ Correspondingly, we will denote
\begin{equation}
	\Lambda(\Gamma) := \overline{\{ {v_a} \ : \ a \in \Gamma \cap S^0 \}} \, \subset  C_1
\end{equation} for the {\em limit set of $\Gamma$ in $C_1$}
and write $V(\Gamma)$ for the linear subspace generated by $\Lambda(\Gamma)$. It is shown below that $\Lambda(\Gamma)$ is $\Gamma$-invariant and thus $V(\Gamma)$ as well.}

\subsubsection{Moment assumptions}
Let $(A_i)_{i \in \N}$ be a sequence of i.i.d. copies of $A$.
 We observe that if $s\ge 0$, the number $$\kappa(s)=\lim_{n\to\8}\E[\norm{A_n\ldots A_1}^s]^{\frac 1n}$$
 is well defined in the interval $$I_{\mu} = \{ s\ge 0; \E [\norm{A}^s]<\8 \}.$$ Indeed,
 since $u_n(s) = \E[\norm{A_n\ldots A_1}^s]$ is submultiplicative,  $u_n^{\frac 1n}(s)$ converges to $\inf_{k\ge 1}\E[\norm{A_k\ldots A_1}^s]^{\frac 1k}$. The function $\log \kappa(s)$
is convex on $I_{\mu}$. If $d=1$, $\kappa(s)$ is the Mellin transform of $\mu$.
% Also, the number $\kappa^*(s) = \lim_{n\to\8} \E[\norm{A^*_n\ldots A^*_1}^s]^{\frac 1n}$
% is well defined under the same  hypotheses and $\kappa^*(s) = \kappa(s)$, since $\norm{A^*}= \norm{A}$.

 As in \cite{KP},  we will assume that $A$
and $N$ have finite expectations: $\E [|A|]+\E [N] <\8$ and we
denote $m=\E[A]\in S$. If
$\rho\in M^1(C_+)$ has a finite mean $\rho_1$ the same is true for
$T\rho$ and the mean $(T\rho)_1$ of $T\rho$ is  $(T\rho)_1 =
\E[N] m\rho_1$. Hence $T$ preserves $M^1_1(C_+)$, the convex subset of
$M^1(C_+)$ of elements with nonzero finite mean. In particular, if a fixed point with finite expectation vector $\rho_1$ exists, then the spectral radius $r(m)$ of $m$ necessarily has to be equal to $\frac1{\E[N]}$.

{ As it will be seen below, condition ($\H$) implies that $m=\E A$ has a
unique dominant eigenvector $v\in C_1^0$ with $m v = r(m) v$. It follows from $\Gamma V(\Gamma) \subset V(\Gamma)$ that $V(\Gamma)$ is $\E A$-invariant and consequently, $v \in V(\Gamma) \cap C_1^0$.}
 The
same is true for $m^*$; let $v^*$ be the unique
eigenvector of $m^*$ in $C_1^*$. It will be shown that $\kappa(1)
= r(m) = r(m^*)$ (Lemma \ref{wcor2.3}).
%The adjoint of $a\in {\rm End} V$ will be denoted $a'$ and the dual
%cone of $C$ will be denoted $C'$, i.e.
%$C'=\{x\in V; \forall y\in C, \is xy\ge 0\}$.

\medskip

%{ In the present paper we study, provided finite expectation of $A$, convergence of iterations of \eqref{iteration}, i.e. convergence of $Y_n$, fixed points of }$T$ in $M_1^1(C_+)$ and
%their tail asymptotics. If $Z$ and $\wh Z$ have the
%same law $\rho$ with finite expectation, by abuse of language, $Z$
%will be called a fixed point of \eqref{iteration}. Clearly, if $Z$
%is a fixed point of \eqref{iteration}, then the law $\rho$ of $Z$
%satisfies $T\rho = \rho$. Also, the condition $T\rho=\rho$ implies
%$\rho_1 = (T\rho)_1 = cm\rho_1$, hence the spectral radius $r(m)$
%of $m=\E A$ is $\frac 1c$. { Therefore, we are going to assume that $r(m)=\frac 1c$ everywhere below.}

\medskip

%We thank Sebastian Mentemeier for very useful remarks.

%\section{Notations and statements of results}
%\label{sec: notations}

%
%We denote by $[{\rm supp}\ \mu]$ the smallest closed
%subsemigroup of $S$ containing ${\rm supp}\ \mu$. We will
%always assume
% that $[{\rm supp}\ \mu]$ satisfies condition ($\H$) and $m=\E[A] = \int a\mu(da)$
% exists.
% Since $d>1$ and any element of $S$ has an
%invariant line in $C$, condition ($\H$) implies that $\mu$
%is not a Dirac measure on $S$.
%{ For $a\in S$ let $r(a)$ be }its spectral radius, i.e. $r(a) = \lim_{n\to\8}|a^n|^{\frac
%1n}$. %Let $\mu\in M^1(S)$ be the law of the random endomorphism $A$.

\medskip

 By abuse of notation the
function $x \mapsto\is {v^*}x$ on $C$ will be denoted also by $v^*$.  We observe that the convolution operator $P$ on $C_+$ defined by
\begin{equation}
\label{defP}
P\phi(x) = \int_S \phi(ax)\mu(da)
\end{equation}
 admits the eigenfunction
$v^*$ with the eigenvalue $r(m)$.
 We will  consider also an analogous Markov operator on $C^*_+$ defined by
$$
P_*\Phi(x) = \int_S \Phi(a^*x)\mu(da).
$$

In the whole paper we will denote by $C$ the cone and sometimes, by abuse of notation, also
continuous functions. We will use the symbol $D$ to denote auxiliary constants, which will appear
in the sequel.

%\footnote{Then we may consider the Markov
%operator $ Q$ on $C_+^* = C_1 \times \R_+^*$ defined by $$
% Q \phi = \frac{1}{r(m) v^*} P(\phi v^*)
%$$ and we will prove below that $Q$ has a unique invariant measure of the form $\pi\otimes %l$
%where $\pi \in M^1(C_1)$ and $l(dt)=\frac {dt}t$.
%}

\subsection{Statement of main results}

Recall the construction of the Mandelbrot's cascade $Y$ described above, which was given in terms of an eigenvector $v$ of $\Erw{A}$.
We obtain the following generalization of the result of Kahane and Peyri\`ere \cite{KP}:
\begin{thm}
\label{existence}
Assume that the semigroup $[{\rm supp}\ \mu]$ satisfies
condition ($\H$), that $\Erw{\norm{A}}<\8$, the dominant
eigenvalue $r(m)$ of $\E[A]$ satisfies $\E[N]r(m)=1$,
$N\ge 2$ a.s. and $\E[N^2]<\8$.
Then the following  are equivalent:
\begin{enumerate}
\item $\E[Y]= v$.
\item $\E[Y] \not= 0 $.
\item There exists a fixed point $Z$ of \eqref{iteration} in $C$  such that $\E|Z|<\8$ and $\E Z\not=0$.
\item $\kappa'(1^-)< 0$.
\end{enumerate}
%Furthermore \eqref{iteration} has a unique fixed point $Z$, such that $|\E [Z]|=1$.
% where $\pi^1$ is the unique probability measure on $C_1$ such that $\wt Q(\pi^1\otimes l) = \pi^1\otimes l$.
 Moreover if $s>1$, $\E\big[|Z|^s\big]<\8$ if and only if $s$ satisfies $\kappa(s)\E[N]< 1$. {\red The random variable $Y$ takes its values in $V(\Gamma)\cap C_+$, $P(Y=0)=0$ and the law of $Y$ is a fixed point of \eqref{iteration} with finite mean.  } If $Z$ satisfies (3), then $Z$ is proportional to $Y$.
\end{thm}
The derivative $\kappa'(1^-)$ can be explicitly computed, however since the formula requires some further definitions, we postpone the details to Section \ref{sec: change} (see Theorem \ref{wthm2.4}).
%If $C =  \R^d_+$, then repeating the argument given in \cite{G} one can prove that \eqref{iteration} has a unique fixed point $Z$, such that
%$|\E Z|=1$.

\medskip

 Let $\rho $ be the law of the fixed point $Z$ obtained in Theorem \ref{existence}.
The following asymptotics of $\rho$ is a generalization of
one dimensional results of Guivarc'h \cite{G} and Liu \cite{L}.
%We observe that $\chi>1$ exists with $c\kappa(\chi)=1$,
%then the operator $\wt Q$ on $C_+$ defined by
%$$
%\wt Q \phi(x) = c\int\frac{\is{v'}{ax}}{\is{v'}x}\phi(ax)\mu(da)
%$$
%has an invariant measure of the form $\pi_{\chi}\otimes \rho^{\chi-1}$
%with $\pi_{\chi}\in M^1(C_1)$. Then the asymptotics of $\rho$ is described by the
%following theorem
\begin{thm}
\label{asymptotic}
Suppose that the hypothesis of Theorem \ref{existence} are satisfied and $Z$ is a fixed point of $T$. Assume additionally that
\begin{enumerate}
\item $[\supp\, \mu]$ is aperiodic,
\item $N\ge 2$ is constant,
\item there exists $\chi>1$  with $\kappa(\chi)=1/N$,
\item $\E\big[ \norm{A^*}^{\chi}\, \abs{\log\norm{A^*}}\big]$, $\E\big[ \norm{A^*}^{\chi}\, \abs{\log \iota(A^*)}\big]$ are both finite,
\item assume
%the law of $Z$ gives zero mass to every affine subspace intersecting the cone $C$.
that for any $r> 0$  and any $u\in C_1^*$, $\P[\is Zu =r]=0$.
\end{enumerate}
Then %\mar{how to prove the radial decomposition?}\mar{if $\pi_{\chi}$ will stay here it must be defined}
for every $x\in C^*_+$
$$
%\lim_{t\to 0} t^{-\chi} (t\cdot \rho) = D\pi_{\chi}\otimes l^{\chi},
\lim_{t\to\8} t^{\chi} P[\is Zx >t] = D(x) >0,
$$ where $D(x)$ is a $\chi$-homogeneous $P_*$-harmonic positive function (i.e. $D(tx) = t^{\chi}D(x)$ and $P_*D(x) = D(x)$), uniquely defined up to a positive coefficient by this property.
\end{thm}
{ For $\Gamma = [{\rm supp} \mu]$, a sufficient condition for the regularity hypothesis (5) to be satisfied will be given in Lemma \ref{inverti}: the action of $\Gamma$  on $V(\Gamma)$ is by invertible linear operators.}
\medskip

Notice, that in view of the result of Boman and Lindskog \cite{BoLi}  Theorem \ref{asymptotic} implies that $\rho$
has multivariate regular variation, i.e. the family of measures $t^{\chi} \delta_t\cdot \rho$, where $\delta_t\cdot \rho$
is the push-forward of the measure $\rho$ by the dilation $x\to tx$ ($t>0$), converges vaguely on $\R^d\setminus \{0\}$ to a $P$-harmonic $\chi$-homogeneous Radon measure. { Notice moreover, that by \cite[Proposition 2.5]{BDMM}, if Assumption $(3)$ is satisfied, the fixed points of $T$ are unique up to scaling.}

\subsection{The structure of the paper}

In the paper we try to follow closely the one dimensional arguments due to Kahane, Peyri\`ere \cite{KP} and to Guivarc'h \cite{G}.  The main difficulty to overcome is that we have to replace scalars by vectors and multiplication by positive numbers by  action of matrices. In the multidimensional setting all the concepts require more intrinsic approach. Thus we are led to introduce many definitions. For the readers convenience, we give a list of symbols in the appendix.

\medskip

A basic tool in \cite{G} is a change of measure, the multivariate analogue of which will be introduced in Section \ref{sec: transfer} (with proofs contained in Section \ref{sec: transfer proof}). Therefore, we have to study operators on $C(C_1)$ related to $P$ and their spectral properties. The limiting function $D(x)$ in Theorem \ref{asymptotic} will for example be given via an eigenfunction of these operators. To prove Theorem \ref{existence}, following \cite{KP}, we  show in Section \ref{sec: mandelbrot} that the Mandelbrot's cascade converges to a nontrivial limit and this limit provides a fixed point of \eqref{iteration}. In Section \ref{sec: change} we prove a strong law of large numbers for products of random matrices, which provides a formula for $\kappa'(1^-)$ and is inter alia needed to check the assumptions of Kesten's renewal theorem in Section \ref{sec: renewal}. Together with the change of measure, this renewal theorem is the main technical ingredient in the proof of Theorem \ref{asymptotic} in Section \ref{sec: proof}. Proofs of some technical lemmas are postponed to the appendix.

%The main tool in the proof of Theorem \ref{asymptotic} is Kesten's renewal theorem \cite{K2} and a change of measure. As in \cite{K}, we have to study a family of operators on $C(C_1)$
%
%
%
%
%
%
%As in \cite{KP, Go} we use the Mellin transform. Here the role of the Mellin transform is played by the function $\kappa$ and the index $\chi>1$ is defined by the equation $\kappa(\chi)=1/N$.
%In Section \ref{sec: transfer} we define a family of transfer operators, which are necessary  to study products of random matrices.
%
%
%
%
%Assuming $\chi$ exists with $\kappa(\chi)=1/N$, we use the number $\chi$ and the transfer operators to change the measure in the probability space (Section \ref{sec: change}). Next we replace the classical results on $\R$: the strong law of large numbers and the renewal theorem by the corresponding results: Theorem \ref{wthm2.4} and Theorem \ref{kesten}. { Finally in Section \ref{sec: proof} we write more precisely the proof of Theorem \ref{asymptotic} and we give all the details.} Some technical lemmas are postponed  to Appendices \ref{appendix} and \ref{app: renewal}.

%\section{Transfer operators}
\section{Change of measure}
\label{sec: transfer}

A fundamental tool in the proofs of the main theorems will be 
a change of measure associated with the relation $\kappa(\chi)=1$.
 Its basic idea is well known: Let $\hat{\mu}$ be the increment law of a (one-dimensional multiplicative) random walk $\hat{S}_n$ with negative drift and let there be $\chi >0$ with $\int \abs{a}^\chi \hat{\mu}(da) = 1$, then the random walk with increment law $\abs{a}^\chi \hat{\mu}(da)$ has a positive drift and is used to prove that $t^\alpha\, \P(\max_n \hat{S}_n > t)$ converges to a positive limit as $t \to \infty$ (see \cite[Example XII.4(b)]{feller}). Just copying this idea will not work for matrices, because if $\int \norm{a}^\chi \, \mu(da)=1$ we only know that $\int \norm{a_2 a_1}^\chi \, \mu^{\otimes 2}(da_1, da_2) \le 1$. Instead, we are going to introduce a change of measure with the help of kernels $(q_n^\chi(x,a))_{n \in \N}$ that behave approximately like $\norm{a}^\chi$ and satisfy for each $x \in C_1$ that
$\int q_n^\chi(x,a_n \dots a_1) \, \mu^{\otimes n}(da_1,\ldots, da_n) =1$.  In this section and the next section, we will develop, following \cite{GL} the theory which is necessary to define these kernels and the change of measure. Here and below, $\mu^{\otimes n}$ stands for the $n$-fold product measure $\mu \otimes \ldots \otimes \mu$ on $S \times \ldots \times S$.

The formula for the kernels will be given in a moment, beforehand, we need to define a class of transfer
operators on $C_1$, the eigenfunctions and spectral properties of which will play a crucial role.
For $s \in I_\mu$, we define a bounded operator on $C(C_1)$ by
$$
P^s\psi(x) = \int_{S}|ax|^s\psi(a\cdot x)\mu(da),
$$ where $a\cdot x = \frac{ax}{|ax|}\in C_1$.
Notice, that the operators $P^s$ are related to the operator $P$ defined in \eqref{defP}.
Namely, if $\phi(x) = \psi(\ov x)|x|^s$ with $ \psi\in C(C_1)$ and $\ov x = \frac{x}{|x|}\in C_1$, then
$P^s \psi(x)  = P \phi(x) $ for $x\in C_1$.
%we get a bounded linear operator $P^s$ on $C(C_1)$ defined by
 We are also going to consider the bounded linear operator $P_*^s$ on $C(C_1^*)$ defined by
$$
P^s_*\psi(x) = \int_{S}|a^*x|^s\psi(a^*\cdot x)\mu(da) = \int_{S}|ax|^s\psi(a\cdot x)\mu^*(da),\qquad \psi\in C(C_1^*),\qquad \psi\in C(C_1^*).
$$
For $\Gamma = {\rm supp} \mu$ both families of operators are well defined due to to property (a) of $(\H)$, while their behavior is governed by property (b) of $(\H)$. In order to get a feeling, consider the simplest case of $\mu$ satisfying $(\H)$, namely the Dirac measure on some $a \in S^0$. It is a consequence of the Birkhoff-Hopf theorem (see \cite[Theorem A.7.1]{LNbook}) that such $a \in S^0$ has an algebraic simple dominant eigenvalue $\lambda_a$ and the corresponding eigenspace is one-dimensional, in particular, there is a unique eigenvector $v_a \in C^0$ with $a v_a = \lambda_a v_a$. It follows that the operator $P^s$ has the algebraic simple dominant eigenvalue $\kappa(s)= \lambda^s$ with the corresponding eigenmeasure being the Dirac measure on $v_a$.

%In general, given a semigroup $\Gamma$ of allowable matrices, set
%\begin{equation}
%	\Lambda(\Gamma) := \overline{\{ v_a \ : \ a \in \Gamma \cap S^0 \}}.
%\end{equation}
% and will play a crucial role in our considerations. These operators
%have nice contraction properties on suitable functional spaces.  {  For us their main property is that $\kappa(s)$ is the dominant eigenvalue both of $P^s$ and $P^s_*$ and the corresponding eigenspaces are one dimensional. This will be stated precisely below as Proposition \ref{wthm2.2}}.
We will prove the following properties of the operators.

\begin{prop}
\label{wthm2.2}
Assume that $\mu\in M^1(S)$ is such that $[{\rm supp\ \mu}]$ satisfies condition ($\H$),
and let $s\in I_{\mu}$. Then it holds that:
\begin{enumerate}
\item The equation $$P^s\psi = \kappa(s)\psi,\qquad  \psi\in C(C_1)$$
has a unique normalized solution $\psi = e^s$ $(|e^s|_\8 = 1)$. The function $e^s$ is strictly positive
and $\ov s$-H\"older with $\ov s = \inf\{1,s\}$.
\item There exists a unique $\nu^s\in M^1(C_1)$ with
$$P^s \nu^s = \kappa(s)\nu^s$$ and we have ${\rm supp}\ \nu^s = \Lambda([{\rm supp\ \mu}])$.
\item The mappings $s \mapsto \es$ and $s \mapsto \nus$ are continuous on $I_\mu$ with respect to the topologies of uniform resp. weak convergence.
\item In the same way there is a unique strictly positive and $\ov s$-H\"older continuous function $\est \in C(C_1^*)$ and a unique $\nust \in M^1(C_1^*)$ such that
 $$P_*^s\est = \kappa(s) \est, \qquad \Pst \nust = \kappa(s) \nust,$$
 and the mappings $s \mapsto \est$ and $s \mapsto \nust$ are continuous.
\item The strictly positive function
\begin{equation}
\label{eq: etilda}
\wt e^s(x) = \int_{C_1^*} \is xy^s \nu_*^s (dy),
\end{equation}
is proportional to $\es$ while $e_*^s(x)$ is proportional to $\int \is xy^s \nu^s(dy)$.
\end{enumerate}
\end{prop}

\begin{rem}
The study of these operators goes back to Kesten: In the proof of \cite[Theorem 3]{K1}, spectral properties of an operator $T_\chi$, which is $P_*^\chi$ in the case $C = \R^d_+$, are studied, but only existence of the eigenfunction, the eigenmeasure and the formula for the spectral radius are proved. In particular the uniqueness, now proved, is crucial when using Markov chain Monte Carlo algorithms to actually calculate $\es$ or $\nus$ (see \cite{BS2009,Janssen2010}).
\end{rem}

\subsection{The change of measure}
Now the change of measure can be introduced. Define
\begin{equation}
\label{qn}
q_n^s(x,a) = \frac{|ax|^s}{\kappa^n(s)}\frac{e^s(a\cdot x)}{e^s(x)}
\end{equation}
and observe that
$$\int q^s_n(x,a_n \dots a_1) \mu^{\otimes n}(da_1, \dots, da_n) = \frac{1}{k(s)^n \es(x)} (\Ps)^n \es (x) = 1.$$
%If $\omega\in \Omega = S^{\N^*}$, we write  $S_n(\omega) = a_n\ldots a_1$,
%$q_n^s(x,\omega) = q_n^s(x,S_n(\omega))$ and we observe that
%$$
%\int q^s_n(x,\omega) \mu^{\otimes n}(d\omega) =1.
%$$
%We denote by $\theta$ the shift on $\Omega$.
Then the system of probability measures $q_n^s(x,\cdot)\mu^{\otimes n}$
is a projective system, hence we can define by the Kolmogorov extension theorem   its projective limit
$\Q_x^s$ on $\Omega=S^{\N}$. We denote by $\E_x^s$ the corresponding
expectation symbol. For $(a_n)_{n \in \N} = \omega \in \Omega$, write $ A_n(\omega) = a_n$ and
$$S_n(\omega) := A_n \cdots A_1 (\omega).$$ If $\E,$ $\P$ are used without any sub-/superscript, then it is always stipulated that $(A_n)_{n \in N}$ is an i.i.d. sequence  having law $\mu$, as it has been used previously.  Then, for all $n \in \N$ and all  measurable $f:S^n \to \R_+ $,
\begin{equation}\label{eq:changeofmeasure}
\E_x^s f(A_1, \dots, A_n) = \frac{1}{k(s)^n \es(x)}\E \left[ \abs{S_n x}^s \es\left(S_n \cdot x \right) f(A_1, \dots, A_n)  \right].
\end{equation}
By (2) of Proposition \ref{wthm2.2}, defining the probability measure $\pi_s$ by $$\pi_s (f) :=  \nus(f \es)/ \nus(\es),$$ we infer that
$$\Q^s = \int \Q_x^s \pi^s(dx)$$ defines a probability measure on $\Omega$ which is invariant w.r.t. the shift
$\theta$ on $\Omega=S^{\N}$.
%It is moreover ergodic due to the uniqueness of $\nus$.

Observe that for all $x \in C_1$, $(S_n \cdot x)_{n \in \N}$ is a Markov chain under $\Qxs$ with transition operator given by
$$ Q^s f(x) = \frac{1}{k(s) \es(x)} \Ps (f \es) (x).$$
%and we observe that $\Q^s$ is shift invariant and ergodic
%since $\pi^s$ is the unique $Q^s$-stationary measure measure on $C_1$. We denote by $\E_x^s$ the corresponding
%expectation symbol.
We adopt the notation for the dual situation and define in an analogous way $Q^{s,*}$, $q^{s,*}$, $\pi^{s,*}$, $\Q^{s,*}$, $\Q_x^{s,*}$, $\E_x^{s,*}$.
Since condition $(\H)$ is is symmetric, all that will be shown below holds as well for the dual counterparts.

This information is sufficient to immediately proceed to the proof of Theorem \ref{existence} in Section \ref{sec: mandelbrot}, i.e. a quick reader may skip the next section where the proof of Proposition \ref{wthm2.2} will be given.

%%%%NEW \\
\section{Properties of transfer operators -- proof of Proposition \ref{wthm2.2}}
\label{sec: transfer proof}

The proof of Proposition \ref{wthm2.2} consists of several steps.  We use  ideas developed in \cite{GL,GLnotes} in a different framework. First, as a general technical prerequisite, we will introduce a metric $b$ on $C_1$ with the main properties that every $a \in S^0$ is a contraction w.r.t. to $b$, and that $b$ is bounded -- in contrast to the usual Birkhoff distance. It will be used in several places where condition (b) of $(\H)$ is applied. Next, we study the set $\Lambda(\Gamma)$, before we turn to the proofs of the spectral radius formula and the existence of $\nus$ and $\es$.
Having those results is enough to define the change of measure and the Markov operator $Q^s$. In fact, we will then study $Q^s$ and prove that it is ergodic with unique invariant measure $\pi^s$, which implies the uniqueness results for $\Ps$. Finally, the uniqueness will be used to prove the continuity assertions for $s \mapsto \es$ and $s \mapsto \nus$.

\subsection{A metric on $C_1$}
Following Hennion \cite{H}, who considered the particular case $C=(\R_+)^d$, we can introduce a variant of Hilbert's cross-ratio metric on $C_1$, which is suitable for studying spectral properties of the operators $P^s$ resp. $P^s_*$. Since its definition plays no role in the subsequent arguments, we will postpone it to the Appendix and only give some properties.

\begin{lem}\label{lem:metric}
There is a metric $b : C_1 \times C_1 \to [0, 1]$ on $C_1$ with the following additional properties:
\begin{enumerate}
\item There is $d>0$ such that $b(x,y) \ge d |x-y|$ for all $x,y \in C_1$,
\item for any compact $K \subset C_1$, $(K,b)$ is a complete metric space, which is homeomorphic to $(K, | \cdot |)$,
\item for all allowable $a$ there is $d(a) \le 1$ such that
\begin{enumerate}
\item $b(a \cdot x, a \cdot y) \le d(a) b(x,y)$,
\item $d(a) <1$ if and only if $a \in S^0$,
\item $d(aa') \le d(a) d(a')$ for all allowable $a'$.
\end{enumerate}
\end{enumerate}
\end{lem}

Then, by the Banach fixed point theorem, every $a \in S^0$ possesses a unique attractive fixed point $v_a \in C_1$. This is an eigenvector for $a$ acting on $V$ and we denote the corresponding eigenvalue by $\lambda_a >0$.

\subsection{The limit set}

For a semigroup $\Gamma$ of allowable matrices, set
\begin{equation}
	\Lambda(\Gamma) := \overline{\{ v_a \ : \ a \in \Gamma \cap S^0 \}}.
\end{equation}

\begin{lem}\label{lem:invariant set}
The set $\Lambda(\Gamma)$ is $\Gamma$-invariant, i.e. $\Gamma \cdot \Lambda (\Gamma) \subset \Lambda(\Gamma)$. Moreover, for $x \in C_1$ the closure of the orbit $\Gamma \cdot x$ contains $\Lambda(\Gamma)$. In particular, $\Lambda(\Gamma)$ is the unique minimal closed $\Gamma$-invariant subset of $C_1$.  Any $\Gamma$-invariant subspace $W$ with $W \cap C_+ \neq \emptyset$ contains $V(\Gamma)$. 

If there is a finite set of subspaces $W_i$ ($i\in I$) with $W_i \cap C_+ \neq \emptyset$, such that each $a \in \Gamma \cap S^0$ permutes these subspaces, then
 each $W_i$ contains $V(\Gamma)$.
\end{lem}

 {\bf Remark}. The last assertion shows the connection with  the irreducibility and proximality condition which is used intensively
for  analogous statements in \cite{GL,GLnotes}. Here we restrict to subspaces which intersect $C_+$. The corresponding concept will be called $C$-strong irreducibility {\red  -- see Lemma \ref{lem:strongirreducible} . There, we will also consider the linear subspace 
$$ V^-(\Gamma) := V(\Gamma^*)^\perp \subsetneq V,$$ i.e. the orthogonal space of $V(\Gamma^*)$, which in turn is the subspace generated by $\Lambda(\Gamma^*)$ with the latter being non-trivial due to hypothesis ($\H$).} 

\begin{proof}
Let $a' \in \Gamma$, $v_a \in \Lambda(\Gamma)$, i.e. $v_a = a \cdot v_a$ for some $a \in \Gamma \cap S^0$ (or there is a sequence $v_{a_n} \to v_a$). Then for every $n$, $a' a^n \in \Gamma \cap S^0$ due to property (a) of $(\H)$, hence $v_{a' a^n} \in \Lambda(\Gamma)$. Applying the properties of $b$, we deduce from
$$ b(v_{a' a^n}, a' \cdot v_a) = b(a' a^n \cdot v_{a' a^n}, a' a^n \cdot v_a) \le d(a') d(a)^n \, b(v_{a' a^n}, v_a)$$
that $v_{a' a^n}$ tends to $v_a$ as $n$ goes to infinity. Since $\Lambda(\Gamma)$ is closed, we infer the $\Gamma$-invariance.

If now $W \neq \emptyset$ is any closed $\Gamma$-invariant subset of $C_1$, we have to prove that $\Lambda(\Gamma) \subset W$. Let $x \in W$, then for all $a \in \Gamma \cap S^0$,  $a^n \cdot x \in W$ for any $n$ due to the invariance of W. But $a^n \cdot x \to v_a$ and the assertion follows since $W$ was assumed to be closed. The same argument shows that $v_a$ is in the closure of the orbit $\Gamma \cdot x$ for any $x \in C_1$. 

{\red For the last assertion notice that if $a \in \Gamma\cap S^0$ permutes the subspaces $W_i$, then for any $i$, for a subsequence, and any $x\in W_i\cap C_1$, we have $a^{n_j} x\in W_i$ and $v_a =\lim_j a^{n_j} \cdot x \in  W_i $. Since these limits generate $\Lambda(\Gamma)$, it follows that $V(\Gamma) \subset W_i$.}
\end{proof}

Then standard arguments from the theory of iterated random Lipschitz functions (see e.g. \cite{DF}) yield the following result:

\begin{lem}
\label{wprop2.1}
%Assume $\Gamma$ is a subsemigroup of $S$ which satisfies condition ($\H$). Then $\Lambda(\Gamma)$
%is the unique $\Gamma$-minimal subset of $C_1$.
If $\mu\in M^1(S)$ is such that $[{\rm supp}\ \mu]$
satisfies ($\H$), then there exists a unique $\mu$-stationary measure $\nu$ on $C_1$ with ${\rm supp}\ \nu = \Lambda(\Gamma)$.
\end{lem}

\subsection{Existence of eigenfunctions}

We are going to prove the following
\begin{prop}\label{prop:existence es}
Assume that $\mu \in M^1(S)$ is such that $[{\rm supp}\, \mu]$ satisfies condition $(\H)$ and let $s \in I_\mu$. Then the spectral radius $r(\Ps)$ of $P^s$ equals $\kappa(s)$ and there are $e^s \in C(C_1)$ and $\nu^s \in M^1(C_1)$ with
$$ P^s e^s = \kappa (s) e^s, \qquad P^s \nu^s = \kappa(s) \nu^s.$$
The function $e_s$ is strictly positive and $\overline{s}$-H\"older with $\overline{s}=\inf \{s,1\}$ and $\supp\ \nus \supset \Lambda(\Gamma)$.

Similarly, the spectral radius of $\Pst$ equals $\kappa(s)$ and there are a strictly positive $\bar{s}$-H\"older function $\est$ on $C_1^*$ and $\nust \in M^1(C_1^*)$
such that
$$ \Pst \est = \kappa (s) \est, \qquad \Pst \nust = \kappa(s) \nust.$$
\end{prop}

Before we are going to prove the proposition, we will need one technical lemma.
\begin{lem}\label{lem:snxsn}
The function $\tau(x) = \inf_{a \in S} \frac{\abs{ax}}{\abs{a}}$ is strictly positive on $C^0$. On every compact subset
$K \subset C^0$, $\inf_{y \in K} \tau(y) > 0$.

If $\nu^*$ is a probability measure on $C_1^*$ with $\supp \nu^* \cap (C^*)^0 \neq \emptyset$, then there is $d_s$ such that for all $a \in S$,
$$ \int_{C^*_1} \abs{a^* y}^s \nu^*(dy) \ge d_s \norm{a}^s. $$
\end{lem}

\begin{proof}
{\sc Step 1.} We observe that the subset of ${\rm End}(V)$, $S^1 = \{ a \in S \, ; \, \norm{a} = 1\}$ is relatively compact, hence its closure $\overline{S^1}$ is compact. If $a \in S^1$, then ${\rm Ker } a \cap C^0 = \emptyset$, hence ${\rm Ker } a \cap C^0  = \emptyset$ if $a \in \overline S^1$. It follows that, if $x \in C^0$, $\abs{ax} > 0$ for any $a \in \overline{S^1}$. Since $a \mapsto \abs{ax}$ is continuous on $\overline{S^1}$, and for any $a \in S$,
$\frac{a}{\norm{a}} \in \overline{S^1}$, it follows $\tau(x) = \inf_{a \in S} \abs{ax} > 0$ is attained. The same argument is valid for the function $(a, x) \mapsto \abs{ax}$ on $S^1\times K$, hence $\inf_{x \in K} \tau(x) > 0$.

{\sc Step 2.} In the same way, one proves that $\tau^*(y)= \inf_{a \in S} \frac{\abs{a^* y}}{\norm{a}}$ is strictly positive on $(C^*)^0$.  Consequently, if the support of $\nu^*$ has nonempty intersection with $C_1^*$, then
$$ \inf_{a \in S} \int_{C^*_1} \frac{\abs{a^* y}^s}{\norm{a}^s} \nu^*(dy) \ge  \int_{C^*_1} \inf_{a \in S} \frac{\abs{a^* y}^s}{\norm{a}^s} \nu^*(dy)  > 0, $$
and the assertion follows.
\end{proof}

\begin{proof}[{Proof of Proposition \ref{prop:existence es}}]
{\sc Step 1.} First we will prove existence and properties of the eigenfunction. We proceed
 as in \cite{K1}. We introduce a self-mapping $\widetilde{\Ps}$ on $M^1(C_1)$ by $\widetilde{\Ps} \nu := \frac{\Ps \nu}{\nu(C_1)}$.
Due to the Schauder-Tychonoff theorem, there is an invariant probability measure $\nu^s$ which becomes an eigenmeasure of $\Ps$. Similarly,
  $P_*^s$ has an eigenmeasure $\nust$ as well. Denote the corresponding eigenvalue of $\Pst$ by $k(s)$.  Set $\Gamma:= [{\rm supp} \mu]$.
Upon defining $$\es(x):= \int_{C_1} \skalar{x,y}^s \, \nust(d y),$$
we see that
\begin{align*} \Ps \es(x) &~= \int_{\Gamma} |a x|^s \int_{C_1} \skalar{a \cdot x, y}^s \, \nu(d y) \ \mu(d a)  ~=~ \int_{\Gamma} \int_{C_1} \skalar{a x, y}^s \, \nu(d y) \ \mu(d a) \\
 &~= \int_{\Gamma} \int_{C_1} \skalar{x, a^* y}^s \, \nu(d y) \ \mu(d a) ~=~  \int_{C_1}  \int_{\Gamma} |a^* y|^s \skalar{x, a^* \cdot y}^s \, \mu(d a) \ \nu(d y) \\
&~= \int_{C_1} \skalar{x, y}^s \, (\nust \Pst) (d y) ~=~\int_{C_1} \skalar{x, y}^s \, k(s) \nust (d y) = k(s) \es(x)
\end{align*}
Of course, $\supp \nust$ is $\Gamma^*$-invariant. Hence, by lemma \ref{lem:invariant set}, $\Lambda(\Gamma^*) \subset \supp \nust$. Since $\Gamma^*$ satisfies $(\H)$ as well,  $\supp \nust \cap (C^*)^0 \neq \emptyset$. Consequently, $\es(x)>0$ for all $x \in C_1$.

That $\es$ is $s$-H\"older with respect to $(C_1, | \cdot |)$ follows from its very definition. But since $b(x,y) \ge d |x-y|$, it follows that
$$ \sup_{x,y \in C_1} \frac{|f(x)-f(y)|}{b(x,y)^{\bar{s}}} \le \sup_{x,y \in C_1} \frac{|f(x)-f(y)|}{d |x-y|^{\bar{s}}} < \infty.$$

\medskip

{\sc Step 2.} Now we consider the spectral radius $r(P^s)$.
Observing that
$$ (\Ps)^n f(x) = \E \Big[\left| A_n \dots A_1 x \right|^s \, f(A_n \dots A_1 \cdot x)\Big] \le |f|_\infty \E \big[ \left\|A_n \cdots A_1 \right\|^s\big], $$
the inequality
$$ k(s) \le r(P^s) \le \lim_{n \to \infty} \Big( \E \big[\left\| A_n \dots A_1 \right\|^s\big]\Big)^{\frac 1n}$$
follows.

Conversely, it suffices to prove that
$  k(s) \ge \lim_{n \to \infty} \big( \E \big[\left\| A_n \dots A_1 \right\|^s\big]\big)^{\frac 1n}.$
Here, we use that $\Pst \nust = k(s) \nust$ and that $\supp \nust \cap (C^*)^0 \neq \emptyset$. Thus, by Lemma \ref{lem:snxsn}
$$ k(s)^n = (\Pst)^n \nus(1) = \int \E \abs{A_n^* \dots A_1^*y}^s \, \nus(dy) =\E  \left[\int  \abs{A_n^* \dots A_1^*y}^s \, \nus(dy) \right] \ge d_s \E \norm{A_n \dots A_1}^s .$$

\medskip

{\sc Step 3.}
We have proven thus far that $r(\Ps)=k(s)=\kappa(s)$ with $\Ps \es = k(s) \es$ and $\Pst \nust = k(s) \nust$.
The same holds true with the roles of $\Ps$ and $\Pst$ interchanged, thus we deduce from $\|A\| = \|A^*\|$ that
 $\kappa(s)=r(\Pst)$ and $\Ps \nus= k(s) \nus$.
\end{proof}

We note the following formula and estimate for $\kappa(s)^n$ resp. $\norm{a}^s$:
\begin{cor}\label{cor:estESn}
There is $d_s >0$ such that for all $n \in \N$, $a \in S$
$$  \kappa(s)^n = \int \E \abs{S_n x}^s \, \nus(dx) \ge d_s \E \norm{S_n}^s$$
and
$$ \norm{a}^s \le \frac{1}{d_s} \int \abs{ax}^s \nus(dx).$$
\end{cor}

%\begin{proof}
%The first identity follows, since $k(s)^n \nus(1) = [(\Ps)^n \nus](1)$. For the inequality, we compute
%$$ \int \E \abs{S_n x}^s \, \nus(dx) = \int \E \left[ \abs{S_n}^s \frac{\abs{S_n x}^s}{\abs{S_n}^s}  \right] \, \nus(dx) = \E \left[ \abs{S_n}^s \int \frac{\abs{S_n x}^s}{\abs{S_n}^s}  \, \nus(dx)\right] $$
%and deduce from Lemma \ref{lem:snxsn} that the integral is bounded from below, independent of $n$---recall that $C^0 \cap \supp \, \nus \neq \emptyset.$ The same argument yields the second estimate.
%\end{proof}

\subsection{Uniqueness of the eigenfunctions and eigenmeasures}

We are going to show that  the Markov operator $Q^s$ defined on $C(C_1)$ by
$$
Q^s \phi(x) := \frac{1}{\kappa(s) e^s} P^s(\phi e^s)(x) $$
 has a unique invariant probability measure $\pi_s$, given by $\pi_s (f) :=  \nus(f \es)/ \nus(\es),$ and that all $Q^s$-invariant functions (from $C(C_1)$) are constant.
This corresponds to the uniqueness (up to scaling) of $\es$ and $\nus$ as eigenfunctions resp. eigenmeasures of $\Ps$, corresponding to the dominant eigenvalue $\kappa(s)$.

%For each $x \in C_1$ let $\Qxs$ be a probability measure such that $(X_n)_{n \in \N}$ is a Markov chain with transition kernel $Q^s$ and initial state $X_0=x$.
%A point $y \in C_1$ is called reachable, if $\sum_{n=1}^\infty \Qxs(X_n \in O) >0$ for all $x \in C_1$ and all open sets $O \subset C_1$ containing $y$.
%It is called topologically recurrent if $\sum_{n=1}^\infty \Q_y^s (X_n \in O) = \infty$ for all open sets $O \subset C_1$ containing $y$. It is called aperiodic, if it is topologically recurrent and for each open set $O \subset C_1$ containing $y$ there is $n(O) \in \N$ such that $\Q_y^s (X_n \in O) >0$ for all $n \ge n(O)$.
%

Therefore, we will use the following result from the theory of Markov chains on general state space.

\begin{thm}[{\cite[Proposition 18.4.4]{MT}, \cite{R}}]\label{thm:meyntweedie}
Let $X$ be a compact space and let $Q : C(X) \to C(X)$ be an equicontinuous Markov operator. If an reachable and aperiodic state exists for the associated Markov chain, then there is a unique invariant probability measure $\pi$ and for all $f \in C(X)$,
$$ \lim_{n \to \infty} Q^n f = \pi(f).$$
Consequently, any $Q$-invariant function is constant.
\end{thm}

Recall, that the measures $\Qxs$ can already be defined -- only existence of $\es$ is needed therefore, and that
$$ X_n := S_n \cdot x$$
defines a Markov chain with transition operator $Q^s$ and initial value $X_0=x$ under $\Qxs$.

Before applying Theorem \ref{thm:meyntweedie}, we should give some definitions  of the terms used there:
We say that $Q^s$ is equicontinuous, if for every $f \in C(X)$ the sequence $\{(Q^s)^n f\}$ is equicontinuous. Then the Markov chain $(X_n)$ is called an e-chain.
For $x \in C_1$, let $\mathfrak{O}(x)$ be the family of open sets $O \subset C_1$ containing $x$. Then $x$ is called
\begin{itemize}
\item \emph{reachable}, if $\sum_{n=1}^\infty \Q_y^s(X_n \in O) >0$ for all $O \in \mathfrak{O}(x)$ and all $y \in C_1$,
\item \emph{topologically recurrent}, if $\sum_{n=1}^\infty \Qxs (X_n \in O) = \infty$ for all $O \in \mathfrak{O}(x)$ and
\item \emph{aperiodic}, if $x$ is topologically recurrent and for each $O \in \mathfrak{O}(x)$ there is $n(O)$ such that $\Qxs(X_n \in O)>0$ for all $n \ge n(O)$.
\end{itemize}

\begin{lem}\label{lem:eqcon}
$Q^s$ is equicontinuous.
\end{lem}

In order to prove the lemma, we will consider the dense subset $H_s \subset C(C_1)$ of $\bar{s}$-H\"older continuous function with respect to the distance $b$.
We denote the H\"older-norm of such functions by
$$ [f]_s := \sup_{x,y \in C_1} \frac{\abs{f(x)- f(y)}}{b(x,y)^{\bar{s}}}.$$

We cite from \cite{GL} the result that the kernels $q_n^s$ are $\bar{s}$-H\"older:

\begin{lem}[{ \cite[Lemma 2.11]{GL}}]\label{lem:prop:qn}
There is $D_s < \infty$ such that for all $n \in \N$, $x,y \in C_1$, $a \in S$
$$ \abs{q_n^s(x,a) - q_n^s(y,a)} \le D_s \frac{\norm{a}^s}{k(s)^n} b(x,y)^{\bar{s}}.$$
\end{lem}

\begin{proof}[{Proof of Lemma \ref{lem:eqcon}}]
Consider first $f \in H_s$. Then
\begin{eqnarray*}
\abs{(\Qs)^n f(x) - (\Qs)^n f(y)}
 &= & \abs{ \E_x^s f(X_n) - \E_y^s f(X_n)} \\
&\le& \E_x^s \abs{ f(S_n \cdot x) - f(S_n \cdot y)  } + \abs{(\E_x^s - \E_y^s) f(S_n \cdot y)}\\ &=& I + II.
\end{eqnarray*}
Considering $I$,
\begin{align*}
I \le [f]_s\, \E_x^s \big[   b(S_n \cdot x, S_n \cdot y)^{\bar{s}}\big] \le [f]_s b(x,y)^{\bar{s}} \ \E_x^s d(S_n)^{\bar{s}} \le [f]_s b(x,y)^{\bar{s}}.
\end{align*}
Turning to $II$, we have, using Corollary \ref{cor:estESn} and Lemma \ref{lem:prop:qn},
$$ II \le \abs{f}_\infty \E \big| q_n^s(x,S_n)  - q_n^s(y,S_n) \big|
\le \frac{\abs{f}_\infty D_s b(x,y)^{\bar{s}}}{k(s)^n} \E \norm{S_n}^s
%\le [f] \frac{C_s}{c_s} d(x,y)^{\bar{s}} \E q_n^s(\mPi_n) \\
\le \abs{f}_\infty \frac{D_s}{d_s} b(x,y)^{\bar{s}}.
$$
Thus we have proven that for all $n \in \N$,
\begin{equation} \abs{(\Qs)^n f(x) - (\Qs)^n f(y)} \le D(f) b(x,y)^{\bar{s}} \label{eq:eqcon}\end{equation}
for some constant depending only on $f$, which shows the equicontinuity of the family ${(\Qs)^n f}$ as soon as $f \in H_s$.
But each $f \in C(C_1)$ can be approximated (w.r.t. $\abs{\cdot}_\infty$) by functions in $H_s$, thus the assertion for general $f$ follows.
\end{proof}

\begin{lem}
There is a reachable and topologically recurrent $v \in C_1$.
\end{lem}

\begin{proof}
In the case where $C$ is the nonnegative cone, this is proven in \cite[p.218-220, proof of I.1]{K1}. Since our result can be proved along similar lines, we only give the basic idea of the proof:
Choose $a \in \Gamma \cap S^0$ with unique attracting fixed point $v_a \in C_1$, i.e. $\lim_{n \to \infty} a^n \cdot x = v_a$ for all $x \in C_1$. Such $a$ exists due to part (b) of condition $(\H)$. Given any initial point $x \in C_1$ and any neighborhood $O$ of $v_a$, there is $n \in N$ such that $a^n \cdot x \in O$. If now $a$ is generated with positive probability, then the same holds for any power $a^n$, which proves that $v_a$ is reachable. If $a$ is a continuity point of the support of $\mu$ resp. of one of its powers $\mu^{\otimes n}$, then small perturbations $a'$ of $a$ still have the property that $(a')^n \cdot x \in O$.

Given a neighborhood $O$ of $v_a$, one can use the compactness of $C_1$ to find $n \in \N$ and $\epsilon >0$ such that $\Qxs(X_n \in O) \ge \epsilon$ for all $x \in C_1$. Then a geometric trials argument yields the topological recurrence of $v_a$.
\end{proof}

Now we are ready to prove the main result of this section.

\begin{thm}\label{thm:Qs ergodic}
The Markov operator $\Qs$ has a unique invariant probability measure $\pi^s$, and for all $\phi \in C(C_1)$,
$$ \lim_{n \to \infty} \Qs \phi = \pi^s(\phi),$$ hence all $\Qs$-invariant functions are constant.
\end{thm}

\begin{proof}
We start with a reduction, called geometric sampling: Defining the Markov operator
$$\overline{\Qs} := \sum_{n=1}^\infty 2^{-n} (\Qs)^n,$$
we observe that every invariant measure or invariant function of $\Qs$ is an invariant measure resp. invariant function for $\overline{\Qs}$. Thus, in order to prove uniqueness, it suffices to show that $\Qs$ satisfies the assumption of Theorem \ref{thm:meyntweedie}. The estimate \eqref{eq:eqcon} remains valid for $\overline{\Qs}$ as well, hence the equicontinuity follows.
Obviously, the existence of a reachable and recurrent state carry over, too. But due to the geometric sampling, any recurrent state is already aperiodic. Thus, Theorem \ref{thm:meyntweedie} applies to the operator $\overline{\Qs}$.
\end{proof}

\begin{cor}
The function $\es$ is unique up to scaling, $\nus$ is the unique  eigenmeasure in $M^1(C_1)$ and $\supp \, \nus = \Lambda([\supp \mu])$.
\end{cor}

\begin{proof}
Recall that $\es$ is strictly positive on $C_1$. If $\Ps \phi = \kappa(s) \phi$, then $\phi/\es$ is $Q_s$-invariant, hence constant. Similar, if $\Ps \rho = \kappa(s) \rho$ for $\rho \in M^1(C_1)$, then
$\es(x)\rho(dx)/\rho(\es)$ is $Q^s$-invariant, hence equal to $\pi^s = \es(x)\nus(dx)/\rho(\es)$, thus $\nus=\rho$ follows.
Finally, observe that $\Qs$ is also well defined as an operator on $C(\Lambda([\supp \mu]))$. Using Schauder-Tychonoff, there is a $\Qs$-invariant measure $\rho$, supported on $\Lambda([\supp \mu])$. By the $[\supp\, \mu]$-invariance of $\Lambda([\supp \mu])$, $\rho$ is as well invariant for $\Qs$ acting on $C(C_1)$. But then, $\rho = \nus$ and consequently, recalling Proposition \ref{prop:existence es},
$$ \Lambda([\supp \mu]) \supset \supp\, \rho = \supp\, \nus \supset \Lambda([\supp \mu]).$$
\end{proof}

\subsection{Continuity of the mappings $s \to \es$, $s \to \nus$}

In the proof of Theorem \ref{existence} we will need continuity of the mapping $s \to \es$. Since we have an explicit formula for $e^1$, this result is of interest in its own right in order to study $e^s$ for $s$ close to $1$.

\begin{prop}\label{prop:continuity e nu}
The mappings $s \to \es$ and $s \to \nus$ are continuous on $I_\mu$ with respect to $\abs{\cdot}_\infty$ resp. to weak convergence. The same holds for $s \mapsto \est$ and $s \mapsto \nust$.
\end{prop}

\begin{proof}
We follow the proof given in \cite{GL}.

{\sc Step 1}. Given $s_0 \in I_\mu$, consider any sequence $s_n \to s_0$, such that the sequence $\nu^{s_n}$ converges to a limit $\eta \in M^1(C_1)$ -- recall that $M^1(C_1)$ is compact w.r.t. the topology of weak convergence. It suffices to show that $\eta=\nu^{s_0}.$ W.l.o.g. let $s_n$ be from a compact subinterval $I \subset I_\mu$.
Due to continuity of $\kappa$ on $I$,
$$ \lim_{n \to \infty} P^{s_n} \nu^{s_n} = \lim_{n \to \infty} \kappa(s_n) \nu^{s_n} = \kappa(s_0) \eta.$$
Then for any $f \in C(C_1)$, we have that $\sup_{s \in I} \abs{\Ps f}_\infty \le \abs{f}_\infty \sup_{s \in I} \E \abs{A_1}^s \le D \abs{f}_\infty$, and consequently,
\begin{align*}
\abs{(P^{s_n} \nu^{s_n})f - (P^{s_0} \eta)f} \le&~\abs{\nu^{s_n}(P^{s_n} f) - \eta(P^{s_n} f)} + \abs{\eta(P^{s_n}f)- \eta(P^{s_0}f)} \\
\le &~ \abs{\nu^{s_n}- \eta}\left( D \abs{f}_\infty \right) + \eta\left( \abs{P^{s_n} f - P^{s_0}f } \right) \to 0.
\end{align*}
Thus $P^{s_0} \eta = \kappa(s_0) \eta$ and consequently, due to uniqueness of $\nu^{s_0}$, $\eta=\nu^{s_0}$. The same calculation shows that $s \mapsto \nu_*^s$ is continuous.

{\sc Step 2}. We use the formula $\es(x) = \int \skalar{x,y}^s \, \nust(dy)$ and note that for each fixed $y \in C_1$ and compact subset $I \subset I_\mu$, the family ${\skalar{\cdot,y}^s}_{s \in I}$ is equicontinuous, thus $\abs{\skalar{\cdot ,y}^s - \skalar{\cdot,y}^{s_0}  }_\infty \to 0$ as $s \to s_0$. We compute
\begin{align*}
\abs{e^s-e^{s_0}}_\infty \le&~ \abs{\int \skalar{\cdot ,y}^s \, (\nu_*^{s} - \nu_*^{s_0})(dy)}_\infty +  \abs{ \int \left(\skalar{\cdot ,y}^s - \skalar{\cdot,y}^{s_0}\right) \, \nu_*^{s_0} }_\infty \\
\le&~ \abs{\nust - \nu_*^{s_0}}(1) + \int \abs{\skalar{\cdot ,y}^s - \skalar{\cdot,y}^{s_0}  }_\infty\, \nu_*^{s_0} \to 0,
\end{align*}
referring to Step 1. A similar proof applies to $s \mapsto \est$.
\end{proof}

\subsection{Calculation of $\kappa(1)$}
%Finally we present a lemma that will be not used in the paper, but it explains the relation between $\kappa(1)$ and $m$.
It has been an open problem ever since to compute the function $\kappa$. The following lemma, proved in the Appendix \ref{appendix}, gives an explicit formula for $\kappa(1)$. This allows to easily check the assumptions of the main theorem, namely $\kappa(1)=1/\E N$ and $\kappa'(1^-)<0$. Due to convexity of $\kappa(s)$ and subadditivity of the norm a sufficient condition for the latter one would be e.g. that $\E \abs{A_1}^s <1/\E N$ for some $s >1$.
\begin{lem}
\label{wcor2.3}
Let $m = \int a\mu(da)$. Then for some $n\ge 1$, $m^n\in S^0$ and if $v^* \in C_1^*$ is the dominant
eigenvector of $m^*$ we have: $\kappa(1) = r(m)$, $e^1(x)=\frac{\is{v^*}x}{|v^*|_\8}$.
\end{lem}

%\subsection{Conclusion}
%
%All in all, we have proved the following result:
%
%\begin{prop}
%\label{wthm2.2}
%Assume that $\mu\in M^1(S)$ is such that $[{\rm supp\ \mu}]$ satisfies condition ($\H$),
%and let $s\in I_{\mu}$. Then the equation $$P^s\psi = \kappa(s)\psi,\qquad  \psi\in C(C_1)$$
%has a unique normalized solution $\psi = e^s$ $(|e^s|_\8 = 1)$. The function $e^s$ is strictly positive
%and $\ov s$-H\"older with $\ov s = \inf\{1,s\}$. There exists a unique $\nu^s\in M^1(C_1)$ with
%$$P^s \nu^s = \kappa(s)\nu^s$$ and we have ${\rm supp}\ \nu^s = \Lambda([{\rm supp\ \mu}])$.
%The mappings $s \mapsto \es$ and $s \mapsto \nus$ are continuous on $I_\mu$ with respect to the topologies of uniform resp. weak convergence.
%
%If
%$\nu^s_*\in M^1(C_1^*)$ satisfies $P_*^s\nu^s_* =\kappa(s)\nu^s_*$, then $e^s(x)$ is proportional
%to $\int \is xy^s \nu_*^s(dy)$.
%
%In the same way the equation $P_*^s\psi = \kappa(s)\psi$ has a unique normalized solution $\psi = e_*^s$ ($|e_*^s|_\8 = 1$) and $e_*^s(x)$ is proportional to $\int \is xy^s \nu^s(dy)$.
%\end{prop}
%
%
%Choosing $\nust$ as a probability measure, we will mainly use
%\begin{equation}
%\label{eq: etilda}
%\wt e^s(x) = \int_{C_1^*} \is xy^s \nu_*^s (dy),
%\end{equation}
%as a special scaled version of the eigenfunction $\es$ (usually, $\abs{\wt e^s}_\infty <1$).
 \subsection{Aperiodicity}

In order to apply Kesten's renewal theorem in the proof of Theorem \ref{asymptotic}, we will have to impose an additional aperiodicity assumption, namely:
\begin{equation} \label{wprop2.2}\tag{A}
\Delta:= \{\lambda_a:a\in\Gamma\cap S^0\} \text{ generates a dense multiplicative subgroup of } \R_+.
\end{equation}

\begin{rem}\label{rem4.1}There is a far from being obvious relation between aperiodicity and invariant subspaces:
%\begin{enumerate}
%\item
It is proved in \cite{GR,GU} that if $\mu$ is supported on a subset of the group ${\rm GL}(V)$ of invertible matrices, then condition $(\H)$ together with strong irreducibility
 and proximality with $d >1$  of $[\supp\, \mu]$ is sufficient for $\Delta$ to generate a dense multiplicative subgroup of $\R_+$.
%Here
%$\Gamma:= [\supp\, \mu]$ is called strongly irreducible, if no finite union $\bigcup_{i=1}^n W_i$ of proper subspaces $W_i \subsetneq V$ satisfies
%\begin{equation}\label{strongirred} \Gamma \left( \bigcup_{i=1}^n W_i \right) \subset \bigcup_{i=1}^n W_i. \end{equation}
%
%{ In our settings of matrices leaving the cone $C$ invariant, strong irreducibility is already implied by the weaker condition that no proper subspace $W$ which non-trivially intersects the cone is $\Gamma$-invariant (see Lemma \ref{lem:strongirreducible}.)}
%\item
%Under assumption $(\H)$, the following general implication can be proved:
%$$ \text{If $\Lambda(\Gamma)$ is uncountable, then $(A)$ holds.}$$
%The following example shows that $(A)$ may fail if $\Lambda(\Gamma)$ is countable: Let $c>0$, $w \in (C_1^*)^0$ and $v_n$ a sequence in $C^0_1$ such that $\is{v_n}{w} = c^n$ for all $n \in \N$ --Write $v_n \otimes w$ for the rank-one matrix given by $v_n \otimes w\,x = \is{w}{x} \, w$ for all $x \in V$. Hence, $v_n \otimes w$ has a dominant eigenvalue $c^n$ with eigenvector $v_n$. Then for $\Gamma= [\{v_n \otimes w \, : \, n \in \N\}]$, we see that $\Lambda(\Gamma)$ is countable and equals the closure $\{ v_n \, ; \, n \in \N \}$
%\end{enumerate}
\end{rem}

\section{Mandelbrot's cascades - Proof of Theorem \ref{existence}}
\label{sec: mandelbrot}

As said before, the main burden of the proof is to show that $Y$ is not identically zero if $\kappa'(1^-) < 0$. In order to extend the approach of \cite{KP} to the multidimensional situation, we are going to consider for $h \le 1$ the quantities $$\Erw{\wt e^h(Y)} = \Erw{ \int_{C_1^*} \is{Y}{u}^h \, \nu_*^h(du)}.$$
On the one hand, $\Erw{\wt e^h(Y)}$ is positive if and only if $\P(Y \neq 0)>0$ while on the other hand, satisfies the identity
$$ \Erw{\wt e^h(AY)} - \Erw{\wt e^h(Y)}= (\kappa(h)-1) \Erw{\wt e^h(Y)}$$
which will in essence be used to link the derivative of $\kappa$ in 1 with properties of $Y$. Therefore, continuity of the mappings $h \mapsto \nu_*^h$, $h \mapsto e_*^h$ will be needed, which was proved in Proposition \ref{prop:continuity e nu}.

\begin{proof}[Proof of Theorem \ref{existence}]

{\sc Step 1.}
%First we describe the Mandelbrot cascade, i.e.
%we construct  a fixed point of \eqref{iteration}.
%Let $\T$ be a tree with root $o$, a typical vertex
%(resp. edge) will be denoted by $\gamma$ (resp. $\a$). We identify
%$\gamma$ with the path from $o$ to $\gamma$, we denote by
%$|\gamma|$ the length of this path and by $N_{\gamma}$ the number
%of children of $\gamma$. If an element $A_\a$ of $S$ is
%associated to each edge $\a$ of $\T$, we obtain a labelled tree
%$\wt \T$. For every vertex $\gamma$ of $\wt \T$, we form the
%product of matrices $P_{\gamma} = A_{\a_1}A_{\a_2}\ldots
%A_{\a_n}$, where $\a_1\a_2\ldots \a_n$ is the path from $o$ to
%$\gamma$. We assume that $A_\a$ (resp. $N_{\gamma}$) are
%independent copies of $A$ (resp. $N$) and that all $A_\a$,
%$N_{\gamma}$ are independent. Let $\F_n$ be the $\s$-field
%generated by all $N_{\gamma}$, $A_\a$ where $|\gamma|\le n$ and
%$\a$ is the edge connecting $\gamma$ to $\gamma'$ being a child of
%$\gamma$ with $|\gamma'|\le n$. If $v\in C_+$ is a dominant
%eigenvector of $\E[A]=m$, we consider the sequence of random
%variables $Y_n = \sum_{|\gamma|=n}P_{\gamma}v$ called Mandelbrot's
%cascade.
{ By Lemma \ref{wcor2.3}, $m=\Erw{A}$ has a dominant eigenvector $v \in C_+$, which is also an element of $V(\Gamma)$ due to Lemma \ref{lem:invariant set}.  Recall from \eqref{mandelbrot cascade} the definition of the Mandelbrot's cascade $Y_n = \sum_{\abs{\gamma}=n} L(\gamma) v$. Then $Y_n \in V(\Gamma)$ for all $n \in \N$.}
%Let $$ N_0 = \# \{i \, ; \, A^i(\emptyset) \neq 0 \},$$ then $N_0$ has the same law as $N$.
For every $i\ \in \N$, we define $Y_n^i$ the shifted version  of $Y_n$,
exactly in the same way as $Y$, but for the subtree rooted at the $i$th child of the root $\emptyset$. Then
$Y_{n+1}=\sum_{i=1}^{\8}A^i(\emptyset) Y_n^i$. By the definition of $A^i$ (see \eqref{defAi})
$$
\E[Y_{n+1}|\F_n] =  \Erw{N} \sum_{|\gamma|=n} L(\gamma) \E[A] v = \Erw{N} r(m) Y_n = Y_n.
$$
{ Then $Y_n$ is a $C_+ \cap V(\Gamma)$-valued martingale, hence $Y_n$ converges
a.e. to $Y\in C \cap V(\Gamma)$. }This follows from the fact there exists a basis such  that the elements of the cone $C$ can be expressed as linear combinations of elements of the basis with positive coefficients and from convergence of positive martingales. We present more details in
Lemma \ref{lem: mtgconvergence} in Appendix \ref{appendix}.

By what has been said above, in the limit we
have
$$
Y = \sum_{i=1}^{\infty} A^i(\emptyset) Y^i =_d \sum_{i=1}^N A_i Y^i
$$ with $Y^i = \lim_{n\to\8}{Y_n^i}$. Next
$\E[Y]\in C$ and $Y_i$'s are independent with the same law as $Y$.
So, if $Y\not= 0 $, we have a solution of \eqref{iteration} and we
are led to discuss below the nondegeneracy of $Y$.

\medskip

{\sc Step 2.}
As an immediate consequence of the construction we obtain equivalence of conditions (1), (2) and (3).
%follows easily from the construction
%of Mandelbrot's cascade.
For the remaining part of the proof we  apply arguments
of Kahane and Peyri\`ere \cite{KP} and Proposition \ref{wthm2.2} to our settings.

\medskip

{\sc Step 3.} We  prove  that (4)  implies (1).
Assume  $\kappa'(1^-)<0$. We will use the following inequality
$$
\bigg(\sum_{i=1}^k y_j  \bigg)^h \ge \sum_{i=1}^k y_j^h - 2(1-h) \sum_{i<j}(y_i y_j)^{\frac h2}
$$
which is valid for $y_i>0$ and $h\in(1-\eps,1]$ for some small $\eps$  independent of $k$ (see \cite{KP}, Lemma C).
Then, we can estimate from below the function $\wt e^h$ defined in \eqref{eq: etilda}. Namely,   for $x_i\in C$, $i=1,...,k$, using Proposition \ref{wthm2.2}
we obtain:
\begin{eqnarray*}
\wt e^h\bigg(\sum_{i=1}^k x_i \bigg) &=& \int_{C_1^*} \bis {\sum_{i=1}^k x_i}u^h\nu^h_*(du)\\
&\ge& \sum_{i=1}^k\wt e^h(x_i)  - 2(1-h)\sum_{i<j} \int_{C_1^*} \is {x_i}u^{\frac h2}\is {x_j}u^{\frac h2}\nu^h_*(du)\\
&\ge& \sum_{i=1}^k\wt e^h(x_i)  - 2(1-h) \sum_{i<j} \bigg(\int_{C_1^*} \is {x_i}u^h\nu^h_*(du)\bigg)^{\frac 12} \bigg(\int_{C_1^*} \is {x_j}u^h\nu^h_*(du)\bigg)^{\frac 12}\\
&=& \sum_{i=1}^k \wt e^h(x_i)  - 2(1-h) \sum_{i<j} \wt e^h(x_i)^{\frac 12} \wt e^h(x_j)^{\frac 12}.
\end{eqnarray*}
Hence, for $Y_n$ and $Y_{n-1}^i$ as defined above,
we have
%\begin{eqnarray*}
$$\wt e^h(Y_n) = \wt e^h\bigg( \sum_{i=1}^N  A_i Y_{n-1}^i  \bigg)
\ge \sum_{i=1}^N \wt e^h\big(A_i Y_{n-1}^i\big)  - 2(1-h)\sum_{i<j} \wt e^h\big(A_i Y_{n-1}^i\big)^{\frac 12}\wt e^h\big(A_j Y_{n-1}^j\big)^{\frac 12},
$$%\end{eqnarray*}
Taking expected value of both sides, we obtain, using Proposition \ref{wthm2.2}:
\begin{eqnarray*}
\E\big[ \wt e^h(Y_n)\big] &=& \E\bigg[ \E\big[ \wt e^h(Y_n)\big| N \big]\bigg] \\
& \ge& \E\Bigg[
\E \bigg[ N \kappa(h) \E\big[ \wt e^h(Y_{n-1})\big]
- N(N-1)(1-h)\bigg(\E\Big[\wt e^h(A_1 Y_{n-1}^1)^{\frac 12}\Big]\bigg)^2\bigg| N\bigg]\Bigg]\\
&=&
\en \kappa(h) \E\big[\wt  e^h(Y_{n-1})\big]
- (1-h) \E\big[N(N-1)\big] \bigg(\E\Big[\wt e^h(A_1 Y_{n-1}^1)^{\frac 12}\Big]\bigg)^2
\end{eqnarray*}
Notice that for every $y\in C_1^*$, $\is {Y_n}y$ is a martingale, so $\is {Y_n}y^h$ is
a supermartingale hence $\E\big[ \is{Y_n}y^h \big]\le \E\big[ \is{Y_{n-1}}y^h \big]$.
Integrating both sides with respect to $\nu^h_*(dy)$ we obtain
 $\E[\wt e^h(Y_n)] \le \E[\wt e^h(Y_{n-1})]$. Therefore
$$\bigg(\E\Big[e^h(A_1 Y_{n-1}^1)^{\frac 12}\Big]\bigg)^2 \ge \frac{ (\E[N]\kappa(h)-1 )\E [e^h(Y_{n-1})]}{\E[N(N-1)](1-h)}$$
and going  with $h$ to the left limit at 1, %(recall that $\kappa$ is differentiable)
we obtain
\begin{eqnarray*}
%\E|A_0|^{\frac 12} \E|Y_{n-1}|^{\frac 12}
 \Bigg(\E\bigg[\bigg(\int_{C_1^*} \is {A_0 Y_{n-1}}u \nu^1_*(du)\bigg)^{\frac 12}\bigg]\Bigg)^2
&=& \bigg(\E\Big[e^1(A_0 Y_{n-1})^{\frac 12}\Big]\bigg)^2\\
&\ge& - \frac{\kappa'(1^-)\E[N]}{\E[N(N-1)]}\cdot
\E\big[ e^1(Y_{n-1})\big]\\ &=&- \frac{\kappa'(1^-)\E[N]}{\E[N(N-1)]} \E\bigg[ \int_{C_1^*} \is{Y_{n-1}}u \nu^1_*(du) \bigg]\\
%&=&  - \frac{\kappa'(1)\E[N]}{\E[N(N-1)]} \int_{S_+} \is{\E Y_{n-1}}u \nu^1(du) \\
 &=&
- \frac{\kappa'(1^-)\E[N]}{\E[N(N-1)]} \int_{C_1^*} \is v u \nu^1_*(du) =|v|^2 \cdot D>0.
\end{eqnarray*}
%Since by Theorem \ref{1.2} measure $\nu^1$ is proper, the constant $C_1$  is strictly positive.

Define now $W_n = \int_{C_1^*} \is{A_0 Y_n}u \nu^1_*(du)$. Then $W_n$ is a positive martingale
and it converges pointwise to $W = \int_{C_1^*}\is{A_0 Y}u\nu^1_*(du)$. Hence the sequence $(W_n)^{\frac 12}$
converges  pointwise to $(W)^{\frac 12}$ and  we will prove that the convergence holds   also in the norm.
For this purpose it is sufficient to observe that the family of random variables $\{(W_n)^{\frac 12}\}$
is uniformly integrable. Indeed since for any positive $x$
$$
x\P[W_n >x] \le \E\big[ W_n\big] = \frac 1{\E[N]} \int_{C_1^*}\is vu \nu^1_*(du)=D_2
$$
and
$$
\E\big[ W_n \big] = P^1 e^1(v) = \frac{e^1(v)}{\E[N]} = D_3
$$
we have
%\begin{eqnarray*}
$$
\lim_{x\to \8}\sup_n \E\Big[ (W_n)^{\frac 12} {\bf 1}_{\{W_n>x\}} \Big] \le  \lim_{x\to \8}\sup_n \Big( \E\big[ W_n\big] \Big)^{\frac 12}
\P[W_n>x]^{\frac 12} \le \lim_{x\to \8} \frac {\sqrt{D_3}\cdot D_2}{\sqrt x} = 0.
%\end{eqnarray*}
$$
Therefore, since $W_n^{\frac 12}$ is a supermartingale bounded in $L^2$,
$$\E\big[ ( W)^{\frac 12}\big] = \lim_{n\to\8} \E\big[(W_n)^{\frac 12}\big] \ge \sqrt{D} $$
and so,  the random variable $Y$ cannot be degenerate.

\bigskip

{\sc Step 4.} Next we prove that (3) implies (4).
We proceed as in \cite{KP} and we use two lemmas proved there: Lemma A and Lemma B, saying that
$$ (x+y)^h \le x^h + h y^h, \quad \mbox{for } x\ge y > 0, 0<h<1
$$
and  for  real valued independent and identically distributed random variables $X,X'$,
there exists  $\eps>0$ such that for all $0<h<1$
$$\E[X^h {\bf 1}_{\{X'\ge X\}}] \ge \eps \E[X^h].$$

Suppose now $Z$ is a  fixed point of \eqref{iteration}.
Let $\{Z_i\}_{i\in\N}$ be a sequence of independent copies of $Z$. We have, using the formulae for $\wt e^h$
\begin{multline*}
\wt e^h(A_1Z_1+A_2Z_2) \le \int_{C_1^*}{\bf 1}_{\{\is {A_1 Z_1}u\le \is {A_2 Z_2}u\}}
(h\is {A_1 Z_1}u^h +\is {A_2 Z_2}u^h) \nu^h_*(du)\\
+ \int_{C_1^*}{\bf 1}_{\{\is {A_1 Z_1}u > \is {A_2 Z_2}u\}}
(\is {A_1 Z_1}u^h +h\is {A_2 Z_2}u^h) \nu^h_*(du).
\end{multline*}
For $h\le 1$ the function $\wt e^h$ is subadditive hence
\begin{eqnarray*}
\E \big[\wt e^h(Z)\big] & = &
\E\Bigg[\E\bigg[ \wt e^h\bigg( \sum_{i=1}^N A_i Z_i \bigg) \bigg| N\bigg]\Bigg]\\
&\le& \E\Bigg[ \sum_{i=3}^N E\Big[ \wt e^h(A_iZ_i) \Big] + \E\Big[\wt  e^h(A_1Z_1+A_2Z_2) \Big]
\bigg| N\Bigg] \\
&\le&
 \en \E\big[ \wt e^h(A_1 Z_1)\big] - 2(1-h)\E \bigg[
\int_{C_1^*}{\bf 1}_{\{\is {A_1 Z_1}u\le \is {A_2 Z_2}u\}} \is {A_1 Z_1}u^h  \nu^h_*(du)\bigg]\\
&\le& \en \kappa(h) \E\big[\wt  e^h(Z)\big] - 2(1-h)\eps \kappa(h) \E\big[\wt  e^h(Z)\big].
\end{eqnarray*}
Hence
$$
2\eps \kappa(h) \E\big[ \wt e^h(Z) \big]\le \frac{\E[N]\kappa(h)-1}{1-h}\cdot \E\big[\wt  e^h(Z)\big]
$$
and passing with $h$ to 1 from below
$$
\frac{2 \eps  \E \big[\wt e^1(Z)\big]}{\E[N]} \le -\E[N]\kappa'(1^-) \E\big[ \wt e^1(Z)\big].
$$
Then, since $\E \big[\wt  e^1(Z)\big]$ is nonzero
%(otherwise $\nu^1$ would not be proper),
$$
\kappa'(1^-) \le -\frac{2\eps}{\E[N]^2}.
$$
Since $\eps>0$ is arbitrary, (4) follows.

\medskip

{\sc Step 5.} As the penultimate step, we have to prove that if $h>1$, then $\E|Z|^h<\8$ if and only if $\kappa(h)\E[N]<1$.

As above be denote by $\{Z_i\}_{i\in \N}$ a sequence of  independent copies of $Z$.
If  $\E|Z|^h<\8$ then also $\E\big[\wt  e^h(Z)\big]<\8$, thus
\begin{eqnarray*}
\E\big[\wt  e^h(Z)\big] &=& \E \Bigg[ \E \bigg[ \int_{C_1} \bis{\sum_{i=1}^N A_i Z_i}u ^h \nu^h_*(du) \bigg| N\bigg] \Bigg]\\
&>&\E\bigg[ \sum_{i=1}^N \E \big[\wt e^h(A_i Z_i)\big]\bigg]
= \en\kappa(h) \E \big[\wt e^h(Z)\big]
\end{eqnarray*}
and since  $\E\big[\wt  e^h(Z)\big]\not= 0$, we deduce $\kappa(h)<\frac 1{\E[N]}$.

\medskip

We omit the converse implication since the argument is exactly the same as in \cite{KP}, p. 137.

\medskip

{\red {\sc Step 6.} We want to prove that $\P(Y=0)=0$. Since we have already seen, that $Y$ takes its values in $V(\Gamma) \cap C$, this will as well imply that in fact, $Y$ is $V(\Gamma) \cap C_+$-valued.  Therefore, we may proceed as in \cite[Theorem 3.2]{DL}:

Since $Y$ is a random variable in $C$, the Laplace transform $\Psi(x):= \E\, e^{- \is xY}$ is finite for all $x \in C^*$. It holds that
$$ \P(Y=0) = \lim_{\abs{x} \to \infty, \, x \in (C^*)^0} \Psi(x) =: \Psi(\infty).$$
In terms of Laplace transform, Eq. \eqref{iteration} becomes
$$ \Psi(x) = \E \, \prod_{i=1}^N \Psi(A_i^* x).$$
Letting $\abs{x}$ tend to infinity while $x$ ranges in $(C^*)^0$, we use that due to assumption $(\H)$, $A^* (C^*)^0 \subset (C^*)^0$, so in the limit,
$$ \Psi(\infty) = \E\,  \prod_{i=1}^N \Psi(\infty) = \E \, \Psi(\infty)^N.$$
Thus, $\P(Y=0)= \Psi(\infty)$ is a fixed point of the function
$ f(t) = \E\, t^N .$ Its only fixed points in the interval $[0,1]$ are 0 and 1, the latter of which is excluded by $\E\, Y \neq 0$, thus $\P(Y=0)=0$.}

Finally, if $Z$ satisfies condition (3) we have, taking conditional expectation with respect to ${\mathcal F}_n$: $Z_n = \sum L(\gamma)w$ with $w\not=0$, eigenvector of $\E A$, hence $w=cv$ with $c>0$. By definition $Z_n$ is a positive martingale, hence from above it converges a.e. to $cY$. Since this limit is $Z$ a.e., we have
$Z=cY$. 
%On the other hand, since $Z$ satisfies equation \eqref{iteration} and for any $i$, $A_i$ preserves $C_+$, the product $Z_+$ of $Z$ by the indicator of 
%$\{ Z\in C_+ \}$ satisfies the same equation. By uniqueness we get $Z= d Z_+$, with $d>0$, hence $Z=Z_+$. This proves that $Z= cY$ does not vanish a.s.

\end{proof}

%\begin{lem}\footnote{refer to AIHP}
%\label{uniqueness}
%If $X$ and $Y$ are two nonzero solutions of \eqref{st} such that
%$|\E X|<\8$ and $|\E Y|<\8$, \footnote{in AIHP the logarithm appears}
%then there exists a  nonzero constant $C$  such that $X = CY$ a.s.
%\end{lem}

\section{Strong law of large numbers}
\label{sec: change}

In this section, we will provide the announced formula for $\kappa'(1^-)$ and moreover, prove a strong law of large numbers for the sequence $\log \norm{S_n}$ under each $\Qxs$.

\begin{thm}
\label{wthm2.4}
Assume that $\mu$ satisfies condition ($\H$), $s\in I_{\mu}$ and $|a|^s\log |a|$, $|a|^s\log \iota(a)$
are $\mu$-integrable. Then, for any $x\in C_1$,
$$
\alpha(s) = \lim_{n\to\8}\frac 1n \log|S_n(\omega)x|= \lim_{n\to\8}\frac 1n \log\norm{S_n(\omega)},\qquad \Q^s \mbox{ a.e. and }\Q_x^s \mbox{ a.e.}
$$
where
\begin{equation}
\label{renstar}
\a(s) = \int_S\int_{C_1} \log|ax| q^s(x,a)\pi^s (dx)\mu(da).
\end{equation}

Furthermore the derivative of $\kappa$ exists and is continuous on $I_{\mu}$.
The  derivative of $\log \kappa(s)$ is finite and given by
$\a(s) = \frac{\kappa'(s)}{\kappa(s)}$. In particular, if $\chi>1$ and $\kappa(\chi)=\kappa(1)$, then $\kappa'(\chi)\in (0,\8)$.

Finally, the derivative  $\kappa'(1^-)$ is given by the following formula
$$
\kappa'(1^-) = \frac 1{r(m)} \int \frac{\is {v^*}{ax}}{\is {v^*}x}\log \is{v^*}{ax}\mu(da)\pi^1(dx).
$$

\end{thm}
Let us mention that by the derivative at an end of a closed interval we will mean half-derivative.

\medskip

The last theorem was proved in \cite{GLnotes} (Theorem 3.11), however in the cone situation its proof is much simpler, therefore we provide below a complete proof.

We start with the following lemma.
\begin{lem}
\label{wlem2.5}
There exists $d_1>0$, such that for any $x\in C_1$: $\Q_x^s\le d_1\Q^s$.
\end{lem}
\begin{proof}
Recalling $D = \sup \{ \es(x) / \es(y) \, : \, x,y \in C_1 \}  < \infty$ and Corollary \ref{cor:estESn},  we have for any $a\in S$:
$$q_n^s(x,a) = \frac 1{\kappa^n(s)} \frac{e^s(a\cdot x)}{e^s(x)} |ax|^s \le \frac D{\kappa^n(s)}|a|^s \le \frac{D^2}{d_s \kappa^n(s)} \int_{C_1} \abs{ay}^s \pi^s(dy) \le d_1 \int_{C_1} q_n^s(y,a) \pi^s(dy)$$
% Since $\pi^s$ is proper, we have, with some $d_3>0$ and $d_1>0$
%\begin{eqnarray*}
%|a|^s &\le& d_3 \int_{C_1} |ay|^s \pi^s(dy),\\
%q^s_n(x,a) &\le& \frac{d_2d_3}{\kappa^n(s)}\int_{C_1}|ay|^s\pi^s(dy) \le d_1\int_{C_1} q_n^s(y,a)\pi^s(dy).
%\end{eqnarray*}
Let $\psi$ be a nonnegative function on $\Omega$ depending of $n$ first coordinates.
Then
\begin{eqnarray*}
\Q_x^s(\psi) &=& \int_{C_1} \psi(\omega) q_n^s(x,S_n(\omega))\mu^{\otimes n}(d\omega)\\
&\le &d_1\int_{\Omega} \int_{C_1} \psi(\omega) q_n^s(y, S_n(\omega))\pi^s(dy)\mu^{\otimes n}(d\omega)\\
&=& d_1 \Q^s(\psi).
\end{eqnarray*}
In view of the arbitrariness of $n$ and $\psi$ the conclusion follows.
\end{proof}

\begin{lem}
\label{wlem2.7}
For any $x\in C_1$ we have
$$\inf_{n\in\N}\frac{|S_n x|}{\norm{S_n}}>0,\qquad \Q^s \mbox{ a.e.} $$
\end{lem}
\begin{proof}
We follow arguments of \cite{K1,H}. Condition ($\H$) implies
%\footnote{$S_n$ i grupa $S$ sie gryza trzeba t ozmienic.}
existence of $k\in\N^*$, and  $a_1,\ldots, a_k\in {\rm supp}\ \mu$
such that their product $a_k\cdots a_1$ belongs to $S^0$.
Since $q_k^s(x,\omega)>0$
we have $\mu^{\otimes k}$ - a.e., $\Q_x^s(S_k(\omega)\in S^0)>0$. Then Lemma \ref{wlem2.5} implies $\Q^s (S_k(\omega)\in S^0)>0$.
Let $\Omega' = \{ \omega \in\Omega; S_k(\omega)\in S^0  \}$,
$T'(\omega) =\inf \{ n\ge 1; \theta^n\omega \in \Omega'\}$,
$T(\omega) = \inf\{m\ge 0; S_m(\omega)\in S^0\}$. Since $S_{n+k}(\omega) = S_k(\theta^n\omega) S_n(\omega)$,
we have $T(\omega)\le k+ T'(\omega)$. Since $\Q^s(\Omega')>0$ and $\theta$ is ergodic with
respect to the invariant measure $\Q^s$, Birkhoff's theorem implies $T'(\omega)<\8$ a.e., hence $T(\omega)<\8$
$\Q^s$ - a.e. If $n\ge T$, we can write
$$S_n(\omega)x = A_n\ldots A_{T+1}A_T\ldots A_1x
$$ hence,
$$
|S_n(\omega)x|\ge \norm{A_n\ldots A_{T+1}}\, \tau(A_T\ldots A_1 x) \ge \norm{S_n(\omega)}\,\frac{\tau(A_T\ldots A_1x)}{\norm{A_T\ldots A_1}},
$$
i.e. by Lemma \ref{lem:snxsn},  since $A_T\ldots A_1x\in C^0$,
$$
\inf_{n \ge T} \frac{|S_n(\omega)x|}{\norm{S_n(\omega)}} \ge \frac{\tau(A_T\ldots A_1x)}{\norm{A_T\ldots A_1}} > 0, \quad \Q^s \mbox{ a.e.}
$$ On the other hand
$$
\inf_{0\le n<T}\frac{|S_n(\omega)x|}{\norm{S_n(\omega)}} >0.
$$ It follows
$$
\inf_{n\in\N}\frac{|S_n(\omega)x|}{\norm{S_n(\omega)}} >0,\qquad \Q^s \mbox{ a.e.}
$$
\end{proof}
\begin{proof}[Proof of Theorem \ref{wthm2.4}]
We consider the space $\wh \Omega = C_1\times \Omega$ and the extended shift $\wh \theta$:
$$
\wh \theta (x,\omega) = (a_1\cdot x,\theta \omega).
$$
Recall that $\wh \Omega$ can be identified with
the space of paths of the Markov chain defined by $Q^s$ and $\pi^s$
is its unique stationary measure. Thus,
 the probability measure
$$\wh\Q^s = \int_{C_1}\int_{\Omega}(\delta_x\otimes \delta_\omega)\Q_x^s(d\omega)\pi^s(dx)
$$ is $\wh \theta$ - invariant and ergodic. We observe that $f(x,\omega) = \log|a_1(\omega) x|$
satisfies $\log\iota(a)\le f(x,\omega)\le \log|a|$,
hence the $\mu$ - integrability of $f(x,\omega)$. Then Birkhoff's theorem gives
$$
\lim_{n\to\8}\frac 1n \log|S_n(\omega) x| = \lim_{n\to\8}\frac 1n \sum_1^\8 f\circ \wh\theta^k(x,\omega)
=\wh \Q^s(f) = \a(s),\qquad \wh\Q^s \mbox{ a.e.}
$$
On the other hand the subadditive ergodic theorem can be applied to the ergodic system $(\Omega,\theta,\Q^s)$
and the sequence $\log\norm{S_n(\omega)}$:
$$
\lim_{n\to\8}\frac 1n \log\norm{S_n(\omega)} = \a_s,\quad \Q^s \mbox{ a.e.,}
$$ where the convergence is also valid in $L^1(\Q^s)$.

For arbitrary $x\in C_1$ and using Lemma \ref{wlem2.7}
$$
\lim_{n\to\8}\frac 1n \log|S_n(\omega)x| = \lim_{n\to\8} \frac 1n \log \norm{S_n(\omega)}
+\lim_{n\to\8} \frac 1n \log \frac{|S_n(\omega)x|}{\norm{S_n(\omega)}} = \a_s,\quad \Q^s \mbox{ a.e.}
$$ Lemma \ref{wlem2.5} gives $\Q_x^s\le d_1 \Q^s$, hence the above convergence is also valid
$\Q_x^s$ a.e., hence $\wh \Q^s$ a.e. Then the above convergences imply $\a_s = \a(s)$, i.e.
the first part of Theorem \ref{wthm2.4}.

In order to prove the second part we show $$\log \kappa(s) = \int_0^s \a(t)dt.
$$
We consider $$v_n(s) = \frac 1n \log \bigg(\int |ax|^s \mu^n(da)\bigg).$$
We observe that
$$
\frac 1{\kappa^n(s)}\int|ax|^s\mu^n(da) = e^s(x)(Q^s)^n( 1/e^s)(x)
$$ is bounded from below and from above. By Theorem \ref{thm:Qs ergodic} this expression has a limit and for any $x\in C_1$
$$
\lim_{n\to\8} \frac 1{\kappa^n(s)}\int|ax|^s\mu^n(da) = e^s(x)\pi^s( 1/e^s)
$$
Thus, taking the logarithm of both sides and dividing by $n$, we obtain
$$
\log\kappa(s) = \lim_{n\to\8} v_n(s).$$ Notice that the last limit does not depend on $x$.
%Furthermore, Proposition \ref{wthm2.2} gives the uniform convergence of $(Q^s)^n\phi$ to %$\pi^s(\phi)$,
%if $\phi\in C(C_1)$. Hence
%$$
%\lim_{n\to\8}\frac 1{\kappa^n(s)}\int |ax|^s\mu^n(da) = e^s(x)\pi^s(1/e^s).
%$$
We write
$$
v_n'(s) = \frac{\frac 1n\frac 1{\kappa^n(s)}\int|ax|^s \log|ax|\mu^n(da)
}{\frac 1{\kappa^n(s)} \int |ax|^s \mu^n(da)
}
$$ and we observe that
$$
\frac 1{\kappa^n(s)} \int|ax|^s \log|ax|\mu^n(da) = e^s(x) \E_x^s\bigg[
\frac 1{e^s(S_n\cdot x)  \log|S_n x|}
\bigg].
$$ From the $L^1(\Q_x^s)$ convergence above we have
$$
\lim_{n\to\8} \int \bigg| \frac 1n \log|S_n(\omega)x| - \a(s)
\bigg| \Q_x^s(d\omega) =0.
$$ Applying again Proposition \ref{wthm2.2} to $\phi\in C(C_1)$:
$$
\lim_{n\to\8} \frac 1n \E_x^s \Big[ \phi(S_n\cdot x)\log|S_n x|\Big] = \a(s)\pi^s(\phi),
$$
hence taking $\phi = \frac 1{e^s}$:
$$
\lim_{n\to\8} \frac 1{n\kappa^n(s)}\int|ax|^s\log|ax| \mu^n(da)
= e^s(x) \a(s) \pi^s(1/e^s).
$$ Therefore, from the expression of $v'_n(s)$
$$
\lim_{n\to\8} v'_n(s) = \a(s).
$$ Clearly $v_n(t)$
is convex and has a continuous derivative on $[0,s]$, hence $v'_n(0) \le v'_n(t) \le v'_n(s) $
and $v_n(s) = \int_0^s v'_n(t)dt$. By dominated convergence we conclude
$$
\log \kappa(s) = \lim_{n\to\8} v_n(s) = \int_0^s \a(t)dt.
$$
The expression of $\a(s)$, the continuity in $s$ of $q^s(x,a)$, $\pi^s$ and the
inequality $\log\iota(a)\le \log |ax|\le \log |a|$ allow to conclude that the left
derivative of $\log \kappa(s)$ is equal to $\a(s)$. Since $\a(s)$
is continuous on $I_{\mu}$, we get also that $\kappa(s)$ has a continuous derivative on $I_{\mu}$,
and $\frac{\kappa'(s)}{\kappa(s)}=\a(s)$, if $s\in[0,s_\8)$. The convexity of $\log \kappa(s)$
gives for $\kappa(\chi)=1$, $\kappa'(\chi)>0$.
 \end{proof}

\section{Kesten's renewal theorem}
\label{sec: renewal}

The
 main tool we will use to prove Theorem \ref{asymptotic} is
Kesten's renewal theorem \cite{K2}. Here we are going to  state it precisely and
check that its assumptions are satisfied in our settings. To avoid introducing some new notation we formulate here all the details in the case the state space $S$ in \cite{K2} is the compact subset $C_1^*$ of $S^{d-1}$ endowed with the metric $d(x,y) = |x-y|$ and the probability space $\Omega$ is endowed with the probability measure $\Q_x^{\chi,*}$.
Recall that due to the symmetry of $(\H)$ w.r.t. $a \mapsto a^*$, all results up to now carry over to the dual counterparts upon replacing $C$ with $C^*$, $e^\chi$ with $e_*^\chi$, $\mu$ with $\mu^*$ etc.
First we introduce some definitions.

\medskip

Fix $x\in C_1^*$ and define $X_0(\omega)=x$. Given $\omega\in \Omega$ we consider its trajectory $S_n(\o)x$ writing it in radial coordinates. Thus we define
\begin{eqnarray*}
X_n(\o)&=&a_n(\o)\cdot X_{n-1}(\o)=S_n(\o)\cdot x,\\
V_n(\o) &=& \log|S_n(\o)x| = \sum_{i=1}^n U_i(\o),
\end{eqnarray*}
for $ U_i(\o) = \log|a_i(\o) X_{i-1}(\o)|$.

We say that a  function $g:C_1^*\times \R \to \R$ is directly Riemann integrable (dRi) if
$g$ is jointly continuous and satisfies
\begin{equation}
\label{dri}
 \sum_{l=-\8}^\8  \sup \Big\{  |h(x,t)|:\; x\in C_1^*, t\in [l,l+1] \Big\}<\8.
\end{equation}
Indeed the definition given in \cite{K2} is formulated in terms of the measure $\Q_x^{\chi,*}$ and on the first sight both seems to be unrelated. It turns out that the above definition is stronger, but since it is  close to the classical definition (see \cite{feller}) is much easier to handle in applications. We postpone to Appendix \ref{app: renewal} further discussions on this condition.

\begin{thm}[\cite{K2}]
\label{kesten}
Assume  the following conditions are satisfied:
\begin{itemize}
\item {\bf Condition I.1}
There exists $\pi^{\chi}_*\in M^1(C_1^*)$ such that $\pi^{\chi}_* Q^{\chi,*}=\pi^{\chi}_*$ and
for every open set $O$ with $\pi(O)>0$,
$\Q^{\chi,*}_x[X_n\in O \mbox{ for some } n]=1$ for every $x\in C_1^*$.
\item {\bf Condition I.2} Let
 $F(dt|x,y)$ be the  conditional law of $U_1$, given $X_0=x$, $X_1=y$, i.e.
$\Q^{\chi,*}_x\big[ X_1\in A, U_1\in B \big] = \int_A\int_B F(dt|x,y)Q^{\chi,*}(x,dy)$.
Then
$$
\int|t|F(dt|x,y) Q^{\chi,*}(x,dy)\pi^{\chi}_*(dx)<\8
$$ and for all $x\in C_1^*$, $ \Q^{\chi,*}_x$ - a.e.:
$$
\lim_{n\to\8} \frac{V_n}{n} = \a = \int t F(dt|x,y)Q^{\chi,*}(x,dy)\pi^{\chi}_*(dx)>0.
$$
\item {\bf Condition I.3}
 There exists a sequence $\{\z_i\}\subset\R$ such that the group
generated by $\z_i$ is dense in $\R$ and such that for each $\z_i$ and
$\l>0$ there exists $y=y(i,\l)\in C_1^*$ with the following property: for each
$\eps>0$ there exists $A\in {\mathcal B}(C_1^*)$ with $\pi^{\chi}_*(A)>0$ and $m_1,m_2\in N$,
$\tau\in\R$ such that for any $x\in A$:
\begin{eqnarray}
\label{d1} &\Q^{\chi,*}_x\big\{ d(X_{m_1},y)<\eps, |V_{m_1} - \tau|\le \l \big\} >0\\
\label{d2} &\Q^{\chi,*}_x\big\{ d(X_{m_2},y)<\eps, |V_{m_2} - \tau-\z_i|\le \l \big\} >0
\end{eqnarray}
\item {\bf Condition I.4}
For each  fixed $x\in C_1^*$, $\eps>0$ there exists $r_0=r_0(x,\eps)$ such that all real valued
$f\in{\mathcal B}((C_1^*\times \R)^{\N})$ and for all $y$ with $d(x,y)<r_0$ one has:
\begin{eqnarray*}
\E_y^{\chi,*} f(X_0,V_0,X_1,V_1,\ldots)&\le&
\E_x^{\chi,*} f^{\eps}(X_0,V_0,X_1,V_1,\ldots) +\eps|f|_\8,\\
\E_x^{\chi,*} f(X_0,V_0,X_1,V_1,\ldots)&\le&
\E_y^{\chi,*} f^{\eps}(X_0,V_0,X_1,V_1,\ldots) +\eps|f|_\8,
\end{eqnarray*}
where $$f^{\eps}(x_0,v_0,x_1,v_1,\ldots)
=\sup\big\{ f(y_0,u_0,y_1,u_1,\ldots):\;
|x_i-y_i|+|v_i-u_i|<\eps \mbox{ if } i\in \N
\big\}$$
\end{itemize}
If a function $g:C_1^*\times\R\mapsto\R$ is %jointly continuous and
directly Riemann integrable,
then for every $x\in C_1^*$
$$
\lim_{t\to\8} \E_x^{\chi,*} \bigg[ \sum_{n=0}^\8 g(X_n,t-V_n)\bigg] = \frac 1{\a}\int_{C_1^*}\pi^{\chi}_*(dy)\int_\R g(y,s)ds.
$$
\end{thm}
%From now \footnote{zmien ten tekst} we will study the measures $\Q_x^{\chi}$ for $\chi>1$ %defined before Theorem \ref{wthm2.4}, i.e. such that $\kappa(\chi)=\frac 1c$. They will play %the role of $\Q^{\chi}_x$ in the Kesten renewal theorem.
% Let $\a = \a(\chi)$ (see \eqref{renstar} for the definition of $\a(\chi)$). The existence %of $\chi$ is assumed in Theorem \ref{asymptotic}.
\begin{prop}
Under hypotheses of Theorem \eqref{asymptotic}, conditions I.1 - I.4 are satisfied  by the measures  $\Q_x^{\chi,*}$.
\end{prop}
\begin{proof}
{\bf Condition I.1:}
Since $C_1^*$ is compact, I.1 can be replaced by uniqueness
of the $Q^{\chi,*}$-stationary measure $\pi^{\chi}_*$. This follows from
the law of large numbers for Markov chains
with a
unique stationary measure:  for any
continuous function $\phi$ on a compact space (see \cite{B}) we have
$$
\lim_{n\to\8} \frac 1n\sum_{k=0}^\8 \phi(X_k) = \pi^{\chi}(\phi) \quad \Q_x^{\chi,*} \mbox{ a.s.}
$$
for all $x\in C_1^*$.
In our situation Proposition \ref{wthm2.2} implies that this condition is satisfied.

\medskip

{\bf Condition I.2:}
 By \eqref{qn}  and the hypotheses of Theorem \ref{asymptotic} we have
 \begin{eqnarray*}
 \int|t|F(dt|x,y)Q^{\chi,*}(x,dy)\pi^{\chi}_*(dx)&=&
  \int  |\log|ax|| q^{\chi,*}(x,a)\mu^*(da)\pi^{\chi}_*(dx)\\
&\le &
D  \int  |\log|ax|| \norm{a}^{\chi}\mu^*(da)\pi^{\chi}_*(dx)\\
&\le &
D  \int \max\big\{ |\log \norm{a}|, |\log\iota(a)|  \big\} \norm{a}^{\chi}\mu^*(da)\pi^{\chi}_*(dx)\\
&= &
D\,  \Erw{\max\big\{ |\log\norm{A^*}|, |\log\iota(A^*)|  \big\} \norm{A^*}^{\chi}}<\8.
\end{eqnarray*}
The second part  follows immediately from Theorem \ref{wthm2.4}.

\medskip

{\bf Condition I.3:}
The aperiodicity condition \eqref{wprop2.2} states that
$$
\D = \{ \log\l_a:\; a\in [\supp\ \mu^*] \cap S^0\} = \{ \log\l_a:\; a\in [\supp\ \mu] \cap S^0\}
$$
is dense in $\R$.
Let $\{\z_i\}$ be a dense, countable subset of $\D$. Fix $\z_i=\log {\l_a}\in \D$,
with  $a=U_1\ldots U_{n}$, $U_j\in{\rm supp}\  \mu^*$ ($1\le j\le n$) and fix $\lambda>0$.
Let $y=y(\z_i,\d) =  v_a\in {\rm supp}\ \nust=\Lambda(\Gamma^*)$ be the dominant eigenvector of $a$.
If $\eps$ is sufficiently small and $B_{\eps}=\{ x\in C_1^*:\; |x -  v_a|\le \eps \}$,
then $a\cdot B_{\eps}\subset B_{\eps'}$ for $\eps'<\eps$ and
$$
\big| \log \l_a - \log|ax| \big| <\l \quad\mbox{ for } x\in B_{\eps}.
$$
Moreover the statement above  remains valid if we replace $a$ by some $a'$ sufficiently close
to $a$. Therefore if $x\in B_{\eps}$ and $n$ as above:
$$
(\mu^*)^{\otimes n}\Big[ |S_{n}\cdot x- v_a|<\eps, \big| \log|S_{n}x| -\log \l_a  \big| <\l   \Big]>0.
$$
By definition,  $\Q_x^{\chi,*}$ restricted to the first $n$ coordinates  is equivalent to $(\mu^*)^{\otimes n}$,
hence \eqref{d2} is satisfied for $m_2=n$ and $\tau=0$; also
\eqref{d1} is satisfied for $m_1=0$, i.e.  by Lemma \ref{wprop2.1},
$v_a\in C_1^*$. We may take $A=B_{\eps}(v_a)$, then if $x\in B_{\eps}(v_a)$
and $\tau=0$
$$
\Q_x^{\chi,*}\big[ |X_0-v_a|<\eps, V_0\le \l \big] >0.
$$
%Summarizing we have just proved that I.3 is satisfied for the measure $\P$.
%However by definition $\Q_x^{\chi}$ is
%equivalent to $\P_x$ on finitely dimensional cylinders and $\pi^{\chi}$
%is equivalent to $\nu$. Hence these measures are positive on the same sets.

\medskip

{\bf Condition I.4:}
The proof is a consequence of  the argument given by Kesten (\cite{K1}, p. 217-218) and Lemma \ref{wlem2.7}.
\end{proof}

\section{Proof of Theorem \ref{asymptotic}}
\label{sec: proof}

\subsection{Sketch of the proof}

Let (the law of) $Z$ be a fixed point of $T$. We want to prove that for all $u \in C_1^*$, the ratio
$ r^\chi \P(\is Z u > r) $ tends to a positive limit $D(u)$ as $r \to \infty$. 
Let
$$f(u,r) =  \int_r^\8 s\, \P(\is Z u \in ds),
$$
then by \cite[Lemma 4.3]{L}, this is equivalent to
$$ r^{\chi -1} f(u,r)  = r^{\chi -1} \Erw{\1_{(r,\8)}(\is Z u) \is Z u} \stackrel{r \to \infty}{\longrightarrow} \frac{\chi}{\chi-1}D(u).$$

In fact, we are going to show that the function
$$
G(u,t) = \frac {e^{t(\chi-1)}}{e_*^{\chi}(u)} \ f(u,e^t), \qquad u\in C_1^*, t\in\R
$$
has a limit for $t \to \infty$, which is positive and independent of $u \in C_1^*$. This obviously implies the convergence above, moreover, $D(\cdot)$ will be proportional to $e_*^\chi(\cdot)$.

%To prove Theorem \ref{asymptotic} we need to consider the action of matrices on the dual
%cone $C^*$. Therefore we are led to define on $\Omega$ measures $\Q_x^{\chi,*}$, using the function $e_*^{\chi}$ instead of the function $e^{\chi}$,   and the corresponding
%expectation symbol $\E_x^{\chi,*}$. Since condition ($\H$) involves $C$ and $C^*$ symmetrically
%all the results proved up to now concerning measures $\Q_x^{\chi}$ hold also for $\Q_x^{\chi,*}$.

%Fix $u\in C_1^*$. Let $X$ be a solution of \eqref{iteration}. We denote by $X_1,..,X_N$ independent copies of $X$ and by
%$A_1,..,A_N$ independent random variables distributed according to $\mu$. Recall that we assume $N$ to be a constant.
%We define two probability measures on $\R_+$,
% $\nu_u$ - the law of $\is Xu$ and
% $\b_u$ - the law of $\sum_{i=2}^N\is{A_iX_i}u$.
%  Since $N\ge 2$,  $A$ is not constant due to the aperiodicity condition and $X$ is not zero, for some $u\in C_1^*$,  $\beta_u$ is not a Dirac measure.
% For $h(t)=t$ , let
% $\eta_u = h \nu_u$ be a positive measure on $\R_+$.
%Denote $\phi(u)=\eta_u(\R_+) = \nu_u(h)$ for $u\in C_1^*$. By Theorem \ref{existence}  the function $\phi$ is well defined.
%One can prove that $\phi$ is an eigenfunction of $P_*^1$, thus $\phi = D e_*^1$ for some positive constant $D$, but this fact will not play any role in our considerations.

\medskip

%\begin{lem}\mar{do we really need this lemma???}
%We have $P_*^1 \phi = \frac 1{N} \phi$, therefore $\phi =  C e^1_*$ for
%some positive constant.
%\end{lem}
%\begin{proof} We write
%\begin{eqnarray*}
%\phi(u) &=& \nu_u(h) = \E\big[\is Xu\big]
%= \E\bigg[  \sum_{i=1}^N \is{A_i X_i}u  \bigg] \\
%&=& N\E\big[\is {AX}u\big] = N \E\big[ \is X{A^*\cdot u} |A^*u| \big]\\
%&=& N\int_{S}\int_{\R} y|a^* u|\nu_{a^*\cdot u}(dy)\mu(da)\\
%&=& N\int_{S}\nu_{a^*\cdot u}(h)|a^* u|\mu(da)\\
%&=& N\int_{S}\phi(a^*\cdot u)(h)|a^* u|\mu(da) = NP^1_*\phi(u),
%\end{eqnarray*}
%therefore by the definition  of $P^1_*$,
%the function $\phi$ is the eigenfunction of $P^1_*$
%corresponding to  $\kappa(1)=\frac 1{N}$ and by uniqueness (Theorem \ref{wthm2.2})
%$\phi = C e_1^*$.
%\end{proof}
%Define $f(u,r)= \eta_u(r,\8)$ and write $r=e^t$. Our aim is to prove that the limit $\lim_{r\to\8} r^{\chi-1}f(u,r)$
%exists. Indeed, we will prove a stronger result. Namely, let
%$$
%G(u,t) = \frac {e^{t(\chi-1)}}{e_*^{\chi}(u)} \ f(u,e^t).
%$$
%Then we have to show that for every $u\in C_1^*$ the limit $$\lim_{t\to\8} G(u,t)$$ exists, is strictly positive and does not depend on $u$.
%

To prove existence of the limit we would like to apply Kesten's renewal theorem (Theorem \ref{kesten})
and for this purpose we will express the function $G$ as a potential of some function $g$ defined on $C_1^*\times \R$, i.e. we write (Lemma \ref{lem5.8})
$$ G(u,t) = \sum_{n=0}^\8 \E_u^{\chi,*} g(X_n,t-V_n).$$

Yet we do not know  whether the function $g$ is directly Riemann integrable and continuous, this is why we proceed as in \cite{Go} and introduce a exponential smoothing: Given a function $h$ on $C_1^*\times \R$ we define the smoothed function $\wt h$ by
\begin{equation}
\label{eq: smooth}
\wt h(u,t) = \int_{-\8}^t e^{s-t} h(u,s)ds.
\end{equation}
It holds that
\begin{equation} \wt G(u,t) = \sum_{n=0}^\8 \E_u^{\chi,*}\, \wt g(X_n,t-V_n), \label{gtilde}\end{equation}
and, by \cite[Lemma 9.3]{Go},
$$ \lim_{t \to \infty} G(u,t) = \lim_{t \to \infty} \wt G(u,t)$$
as soon as the right hand side exists.
In order to apply Kesten's renewal theorem to prove the existence of the limit of \eqref{gtilde} for $t \to \infty$, it remains to show that $\wt g$ is directly Riemann integrable and continuous. This is the content of the Lemma \ref{lem5.16} and Lemma \ref{lem5.18}, upon which the proof of Theorem \ref{asymptotic} will be finished. Positivity of the limit follows immediately, as the function $\wt g$ will be nonnegative and not identically zero.

\subsection{$G$ is a potential}

Let $(Z_i)_{i=1}^N$ be i.i.d. copies of the fixed point $Z$, and write $\b_u$ for the law of $\sum_{i=2}^N A_i Z_i$. Observe that $\b_u$ is not a Dirac measure due to independence and the aperiodicity condition $(A)$.

%However the function $f$ is not sufficiently 'smooth' to use directly Kesten's theorem,
%therefore we consider its smoothed version:
%\begin{equation}
%\label{tildeG}
% G(u,t)  =  \frac 1{e    ^t e^\chi(u)}\int_0^{e^t} r^{\chi-1} f(u,r)dr.
%\end{equation}
%In consecutive lemmas below we will prove that $G$ is a potential of a continuous and   Riemann
%directly integrable function $g$.
\begin{lem}
\label{lem5.8}
We have
\begin{equation}
\label{potential}
 G(u,t) = \sum_{n=0}^\8 \E_u^{\chi,*} g(X_n,t-V_n),
\end{equation}
for
\begin{equation}\label{gut}
g(u,t) =
\frac {N}{ e^{\chi}_*(u)}\int_{\R^*_+} e^{t({\chi}-1)} \E\Big[
{\bf 1}_{(e^t-y,e^t)}(\is{AZ}u) \is{AZ}u\Big]\b_u(dy),
\end{equation}
%\begin{equation}\label{gut}
%g(u,t) =
%\frac {N}{e^t e^{\chi}(u)}\int_0^{e^t}\int_{\R^*_+} r^{{\chi}-1} \E\Big[
%{\bf 1}_{(r-y,r)}(\is{AX}u) \is{AX}u\Big]\b_u(dy)dr.
%\end{equation}
 where the non negative function g is not identically zero.
%In particular the function $g$ is nonnegative.
\end{lem}
\begin{proof}
{\sc Step 1, definition of $g$}.
On the set of measurable function on  $C_1^*\times \R$ we define the Markov operator $\Theta$ by
$$
\Theta h(u,t) = \E_u^{{\chi},*}\big[ h(X_1, t- V_1)\big]
$$
and let
\begin{equation}
\label{defg}
g(u,t) = G(u,t) - \Theta G(u,t).
\end{equation}
We will prove below that $g$ satisfies \eqref{gut}. Notice
first that
\begin{equation}
\label{sss}
\begin{split}
\Theta G(u,t) &= \E_u^{{\chi},*}\Big[  G(A_1 \cdot u, t-\log|A_1u|) \Big]\\
&= \E\bigg[\frac 1{e^{\chi}_*(A^*\cdot u)} \frac{e^{t(\chi-1)}}{|A^*u|^{\chi-1}}
f\Big( A^*\cdot u,\frac {e^t}{|A^* u|}  \Big)\ \frac 1{\kappa(\chi)}
|A^* u|^{\chi}\frac{e^{\chi}_*(A^*\cdot u)}{e^{\chi}_*(u)}
\bigg]\\
&= \frac {N e^{t(\chi-1)}}{e^{\chi}_*(u)} \E\Big[
{\bf 1}_{(\frac{e^t}{|A^*u|},\8)}(\is{Z}{A^*\cdot u}) \is{Z}{A^*\cdot u}|A^*u|\Big]\\
&=\frac {N e^{t(\chi-1)}}{e^{\chi}_*(u)} \E\Big[
{\bf 1}_{(e^t,\8)}(\is{AZ}u) \is{AZ}u\Big].
\end{split}
\end{equation}
Therefore, since the law of the pair $(\is Zu, \is{A_iZ_i}u)$ is independent of $i$,
\begin{equation}
\label{ssss}
\begin{split}
 G(u,t) &=  \frac 1{ e^{\chi}_*(u)} e^{t({\chi}-1)} f(u,e^t)\\
 &=  \frac 1{ e^{\chi}_*(u)}e^{t({\chi}-1)}
  \E\Big[{\bf 1}_{(e^t,\8)}(\is{Z}u) \is{Z}u\Big]\\
 &=  \frac {N}{ e^{\chi}_*(u)}e^{t({\chi}-1)}
  \E\bigg[{\bf 1}_{(e^t,\8)}\bigg(\bis{\sum_{i=1}^N A_iZ_i}u\bigg) \is{A_1Z_1}u\bigg]\\
 &=  \frac {N}{ e^{\chi}_*(u)} \int_0^\8 e^{t({\chi}-1)}
  \E\Big[{\bf 1}_{(e^t,\8)}(\is{A_1Z_1}u+y) \is{A_1Z_1}u\Big]\b_u(dy).
  \end{split}
\end{equation}
Combining \eqref{defg} with \eqref{sss} and \eqref{ssss} we obtain \eqref{gut}. Furthermore, since for some $u$, $\beta_u$  is not a Dirac measure, the function $g$ is not identically zero.

\medskip

{\sc Step 2.} Iterating the equation \eqref{defg},  we obtain
\begin{equation}
\label{ren}
 G(u,t) = \Theta^n  G(u,t) + g(u,t) + \Theta g(u,t) +\cdots+ \Theta^{n-1}(g)(u,t).
\end{equation}
Therefore to prove \eqref{potential} it is enough to show that $\Theta^n G$ converges to 0 as $n$ goes to $\8$.
Notice
\begin{eqnarray*}
\Theta^n  G(u,t) &=& \E_u^{{\chi},*}\Big[  G(X_n,t-V_n) \Big]\\
&=& \E\bigg[ G(X_n,t-V_n)\ \frac{e^{\chi}_*(S_n\cdot u)}{{\kappa^n({\chi})}e^{\chi}_*(u)} |S_nu|^{\chi} \bigg]\\
&=& \frac{N^n}{e^{\chi}_*(u)}\E\bigg[  \frac{e^{t(\chi-1)}}{e^{\chi}_*(S_n\cdot u)|S_n u|^{\chi-1}}\ {\bf 1}_{(\frac {e^t}{|S_nu|},\8)}
(\is Z{X_n})\is Z{X_n}\ e^{\chi}_*(S_n\cdot u)|S_nu|^{\chi}\bigg]\\
&=& \frac{N^n e^{t(\chi-1)}}{e^{\chi}_*(u)}\E\bigg[ {\bf 1}_{({e^t},\8)} (\is {S_n u}u)\is {S_n u}u\bigg]
\end{eqnarray*}
Let us estimate the expected value. Choose a positive $p$ satisfying $\max\{1,{\chi}-1/2\}<p<{\chi}$. Then
for  $\eps<\frac 1N$, by independence
of $S_n$ and $Z$ we have
\begin{equation}
\label{ineq}
\begin{split}
 \E \Big[{\bf 1}_{(e^t,\8)}(\is {S_n u}u)\is {S_n u}u \Big]
 &\le \sum_{k=0}^\8 \P\Big[ e^t 2^k \le \is{S_n u}u \le e^t 2^{k+1}\Big]e^t 2^{k+1}\\
 &\le \sum_{k=0}^\8 \P\Big[ \is{S_n u}u \ge e^t 2^{k}\Big]e^t 2^{k+1}\\
 &\le \sum_{k=0}^\8  \frac{e^t 2^{k+1}}{e^{pt} 2^{pk}} \E\big[\norm{S_n}^p\big]\E\big[|Z|^p\big]\\
 &\le D e^{(1-p)t}\Big(\frac 1N -\eps\Big)^n\E\big[|Z|^p\big],
\end{split}
\end{equation}
(we use here the definition of $\kappa$ and Theorem \ref{existence}).
Therefore for fixed $t$ and $u$
$$
\lim_{n\to\8} \Theta^n  G(u,t) \le
\lim_{n\to\8}   D  N^n \Big(\frac 1N -\eps\Big)^n e^{(1-p)t}e^{(\chi-1)t}=0
$$
and letting  $n$ go to infinity in \eqref{ren} we
  get the formula for $g$.
\end{proof}

Next, we apply the smoothing operator \eqref{eq: smooth} and define  $\wt g$ and $\wt G$. It follows immediately from \eqref{potential} that
 $\wt G$ is the potential of $\wt g$:
$$ \wt G(u,t) = \sum_{n=0}^\8 \E_u^{\chi,*} \wt g(X_n,t-V_n).$$
In the following  Lemmas  we prove  that $\wt g$ is continuous and directly Riemann integrable so that we will be able to apply  Kesten's renewal theorem.

We will use here the following formula for $\wt g$, which is an immediate consequence of \eqref{gut} and the definition of the smoothing operator,
$$
\wt g(u,t) =
\frac {N}{e^t e^{\chi}_*(u)}\int_0^{e^t}\int_{\R^*_+} r^{{\chi}-1} \E\Big[
{\bf 1}_{(r-y,r)}(\is{AZ}u) \is{AZ}u\Big]\b_u(dy)dr.
$$
\medskip

%Now we prove continuity of $\wt g$:
%{ CHECK!!!!}
\begin{lem}
\label{lem5.16}
Under hypotheses of Theorem $\ref{asymptotic}$ the non negative function $\wt g$ is jointly continuous  and not identically zero.
\end{lem}
\begin{proof}
Since for every positive $r$ and $u\in C_1^*$, $\P[\is Zu=r]=0$, it follows that the function $u\mapsto \E\big[ {\bf 1}_{(r,\8)}(\is Zu)\is Zu \big]$ is continuous. Then continuity of $\wt g$ follows immediately from the formula above.
 Lemma \ref{lem5.8} implies that $\wt g$ is not identically zero.
\end{proof}

Now we give a sufficient condition for hypotheses (5) of Theorem \ref{asymptotic} to be satisfied. The proof is given in the appendix.
\begin{lemma}\label{inverti}
{ Assume additionally that $Z \in V(\Gamma) \cap C_+$  and that the restriction of $\supp\, \mu$ to $V(\Gamma)$ consists of invertible matrices.
}
Then  \begin{equation}
\label{eq: zero}
\P[\is Zu =r]=0.
\end{equation} for all $u \in C_1^*$ and all $r \ge 0$.

 In particular for $Z=Y$ and $r \ge 0 $ we have $\P[\is Y u =r]=0$. 
\end{lemma}

%The next lemma says that because $g$ is integrable, its smoothed version $\wt g$ is directly Riemann integrable.

The next lemma is an analog of Lemma 9.1 in \cite{Go}, which cannot be used in our settings, since our definition of direct Riemann integrability   involves uniform estimates with respect to $u$.
\begin{lem}
\label{lem5.18}
The function $\wt g$ is directly Riemann integrable.
\end{lem}
\begin{proof}{\sc Step 1}.
Recall that $\wt g$ is nonnegative.
In order for the summability condition \eqref{dri} to be satisfied, it suffices to prove
$$
\wt g(u,t) \le D e^{-\eps |t|}.
$$
For negative $t$ we just write
$$
\wt g(u,t) \le \frac D{e^t}\int_0^{e^t} r^{{\chi}-1}dr \E\big[\is{AZ}u\big] \le D e^{({\chi}-1)t} = D e^{-(\chi-1)|t|}.
$$

For positive $t$ we  fix $p$ very close to ${\chi}$ ($1<p<{\chi}$)  and $\eta$ satisfying
$$
1-\frac 1{\chi} < \eta < \min\Big\{ \frac 1{{\chi}+1-p}, 1+p-{\chi}   \Big\}
$$
(in particular $\eta<1$).
We first estimate
$$
\wt g(u,t) \le D \big( g_1(u,t)+ g_2(u,t) \big),
$$ for
\begin{eqnarray*}
g_1(u,t) &=&
\frac 1{e^t}\int_0^{e^t}\int_{\frac {e^{\eta t}}2}^\8 r^{{\chi}-1} \E\Big[
{\bf 1}_{(r-y,r)}(\is{AZ}u) \is{AZ}u\Big]\b_u(dy)dr,\\
g_2(u,t) &=&
\frac 1{e^t}\int_0^{e^t}\int_0^{\frac {e^{\eta t}}2} r^{{\chi}-1} \E\Big[
{\bf 1}_{(r-y,r)}(\is{AZ}u) \is{AZ}u\Big]\b_u(dy)dr.
\end{eqnarray*}

\medskip
{\sc Step 2, estimation of $g_1$.}
Then
\begin{eqnarray*}
g_1(u,t) &\le& \frac 1{e^t}\int_0^{e^t} r^{{\chi}-1}dr\b_u\Big\{y:\; y>\frac{e^{\eta t}}2\Big\}
\E\big[ \is{AZ}u\big]\\
&\le& D e^{({\chi}-1)t} e^{-q\eta t}\int_{\R_+} y^q \b_u(dy),
\end{eqnarray*}
hence if we choose $q\in(\frac{{\chi}-1}{\eta},{\chi})$ then the expression above
can be estimated by $D e^{-\eps_1 t}$ for some $\eps_1>0$. Indeed the last integral is finite
since by Theorem \ref{existence} we have
$$
\int_\R y^q \b_u(dy) \le \E\bigg[ \bis{\sum_{i=2}^N A_i Z_i}u^q\bigg]
\le \E\Big[ \is Xu^q \Big] <\8
$$

\medskip
{\sc Step 3, estimation of $g_2$.}
To estimate $g_2$ we fix $y<\frac{e^{\eta t}}2$.
We will first prove
\begin{equation}
\label{ww}
\frac 1{e^t}\int_0^{e^t} r^{{\chi}-1} \E\Big[{\bf 1}_{(r-y,r)}(\is{AZ}u) \is{AZ}u\Big]dr
\le D e^{-\eps_2 t} +  De^{-t}y^{\chi}.
\end{equation}
 We will use the inequality
$$
\E\Big[ {\bf 1}_{(r,\8)}(\is{AZ}u) \is{AZ}u \Big] \le D r^{1-p},
$$ which was proved in the previous lemma (compare \eqref{ineq}).

By \eqref{ineq} we have
\begin{multline*}
\frac 1{e^t}\int_0^{e^t} r^{{\chi}-1} \E\Big[{\bf 1}_{(r-y,r)}(\is{AZ}u) \is{AZ}u\Big]dr\\
=\frac 1{e^t}\int_y^{e^t} r^{{\chi}-1} \E\Big[{\bf 1}_{(r-y,\8)}(\is{AZ}u) \is{AZ}u\Big]dr
+\frac 1{e^t}\int_0^{y} r^{{\chi}-1} \E\Big[{\bf 1}_{(0,\8)}(\is{AZ}u) \is{AZ}u\Big]dr\\
-\frac 1{e^t}\int_0^{e^t-y} r^{{\chi}-1} \E\Big[{\bf 1}_{(r,\8)}(\is{AZ}u) \is{AZ}u\Big]dr
-\frac 1{e^t}\int_{e^t-y}^{e^t} r^{{\chi}-1} \E\Big[{\bf 1}_{(r,\8)}(\is{AZ}u) \is{AZ}u\Big]dr\\
\le
\frac 1{e^t}\int_y^{e^t}\Big( r^{{\chi}-1} -(r-y)^{{\chi}-1} \Big) \E\Big[{\bf 1}_{(r-y,\8)}(\is{AZ}u) \is{AZ}u\Big]dr
+\frac 1{e^t}\int_0^{y} r^{{\chi}-1} dr \ \E\big[\is{AZ}u\big] \\
\le \frac D{e^t}\int_y^{e^t} r^{{\chi}-1} \Big(1 - \Big( 1-\frac yr \Big)^{{\chi}-1}\Big) (r-y)^{1-p}dr + De^{-t}y^{\chi}
\end{multline*}
To estimate the first integral we divide it into two parts:
the integral over the interval $(y,2y)$ and the second one over $(2y,e^t)$.
We study each of them separately.

\medskip
{\sc Step 3a}.
To estimate the first integral we write:
\begin{eqnarray*}
\frac 1{e^t}\int_y^{2y} r^{{\chi}-1} \Big(1 - \Big( 1-\frac yr \Big)^{{\chi}-1}\Big) (r-y)^{1-p}dr
&\le& \frac {2^{{\chi}-1}y^{{\chi}-1}}{e^t}\int_y^{2y}  (r-y)^{1-p}dr\\
&\le& D e^{-t} y^{{\chi}+1-p} \le D e^{(\eta({\chi}+1-p)-1)t} \le  D e^{-\eps_2 t}.
\end{eqnarray*}

\medskip
{\sc Step 3b}.
To handle with the second one we will use the following inequality, valid for $0\le a\le \frac 12$,
being a consequence of the mean value theorem:
$$
1-(1-a)^{{\chi}-1} \le  D a,
$$ for some constant $D$ depending only on $\chi$.
We have:
\begin{eqnarray*}
\frac 1{e^t}\int_{2y}^{e^t} r^{{\chi}-1} \Big(1 - \Big( 1-\frac yr \Big)^{{\chi}-1}\Big) (r-y)^{1-p}dr
&\le& \frac D{e^t}\int_{2y}^{e^t} r^{{\chi}-1} \frac yr (r-y)^{1-p}dr\\
&\le& \frac{D y}{e^t} \int_{2y}^{e^t} r^{{\chi}-1 - p}dt \le Dy e^{({\chi}-p-1)t}\\
&\le& D e^{(\eta+{\chi}-p-1)t} \le D e^{-\eps_2t}
\end{eqnarray*}
Thus, we obtain \eqref{ww}.
Finally o estimate $g_2$ we choose $\eps < \min\{ {\chi}-1,1 \}$ and write
\begin{eqnarray*}
g_2(u,t) &\le& \int _0^{e^{\eta t}} \Big( D e^{-\eps_2 t} +  De^{-t}y^{\chi} \Big)\b_u(dy)\\
&\le& D e^{-\eps_2 t} + D e^{(\eps\eta-1)t}\int_{\R_+} y^{{\chi}-\eps}\b_u(dy)
\le D e^{-\eps_3 t}.
\end{eqnarray*}
\end{proof}

\subsection{Proof}

Now we can finish the proof:

\begin{proof}[Proof of Theorem \ref{asymptotic}]
 By  Kesten's
renewal theorem \cite{K2}, and using Lemmas \ref{lem5.16} and \ref{lem5.18},
$$
\lim_{t\to\8} \wt G(u,t) = \frac 1{\a({\chi})}\int_{C^*_1}\int_\R \wt g(y,s)ds \pi^{\chi,*}(dy) = D_+ > 0.
$$
Hence, by definition of $G$ and $\wt G$,
$$
\lim_{t\to\8} \frac 1{e^t}\int_0^{e^t} r^{\chi-1} f(u,r) dr = e_*^{\chi}(u) D_+.
$$
'Unsmoothing' the function $\wt G$ (Goldie \cite{Go}, Lemma 9.3) we obtain
$$
\lim_{t\to\8} t^{\chi-1} \eta_u(t,\8) = \lim_{t\to\8} t^{\chi-1} f(u,t) = e^{\chi}_*(u) D_+.
$$
Finally by \cite[Lemma 4.3]{L} we deduce
%Theorem 9.1, Chapter VIII of Feller \cite{feller} we deduce
$$
\lim_{t\to\8} t^{\chi} \P[\is Zu>t] = \lim_{t\to\8} t^{\chi} \nu_u(t,\8) = e^{\chi}_*(u)\frac{(\chi-1) D_+}{\chi}.
$$
Finally replacing $u \in C_1^*$ by any $v\in C^*$ we obtain the result.
\end{proof}

\appendix

\section{The Birkhoff distance on $C_1$}
Following Hennion \cite{H}, we introduce a distance on $C_1$, which is such that on the one hand, every $a \in S^0$ is a contraction w.r.t. this distance, on the other hand, it is compatible with the norm on $V$, restricted to $C_1$. It is defined (a bit heuristically) as follows:

Given $x \neq y \in C_1$, consider the line $L$ through these points (which does not contain $0$). Then, since $C$ is closed and convex, the intersection of $L$ with the boundary of $C$ consists of exactly two points $a,b$ and we define the orientation on $L$ in such a way that (on $L$) $a \le x \le y \le b$. Writing
$x=u_1 a + u_2 c$ and $y = w_1 a + w_2 c$, we have $u_1 \ge w_1 \ge 0$ as well as $w_2 \ge u_2 \ge 0$ and the cross-ratio
\begin{equation}
[a,c;x,y] = \frac{u_1 w_2}{u_2 w_1} \ \in \ [0,1],
\end{equation}
with the cross-ratio being equal to 0 (to 1) iff $x$ or $y$ are extremal points of $C$ resp. iff $x=y$.

The formulae
\begin{equation}
b(x,y) := \phi([a,c;x,y])
\end{equation}
with $\phi(s)=\frac{1-s}{1+s}$ then defines a bounded distance on $C_1$. One can follow the proof of \cite[Lemma 10.4]{H} to see that
\begin{equation}
b(x,y) \ge d | x-y|
\end{equation}
for some $d>0$ and all $x,y \in C_1$.

The distance $b$ is directly connected to Hilbert's cross-ratio metric $d_H$ by the formula
$$ b(x,y)=\tanh (\frac12 d_H(x,y)).$$
By \cite[Corollary 2.5.6]{LNbook}, the topologies on $C_1^0=C_1 \cap C^0$ generated by $d_H$ resp. the norm on $V$ coincide. In particular, the image $a C_1$ is a compact subset of $C_1^0$ if $a \in S^0$. Then one can follow the proof of \cite[Lemma 10.5]{H} to obtain:

\begin{prop}
For $a \in S$ there exists $d(a) \le 1$ such that:
\begin{enumerate}
\item $b(a \cdot x, a \cdot y) \le d(a) b(x,y)$,
\item $d(a)<1$ if and only if $a \in S^0$,
\item if $a' \in S$, then $d(a a') \le d(a) d(a').$
\end{enumerate}
\end{prop}

It follows from \cite[Proposition 2.5.4]{LNbook} that for each compact subset $K \subset C_1$, $(K,b)$ is a complete metric space. Hence, Banach's fixed point theorem applies for $a \in S^0$ and we deduce existence and uniqueness of a attractive fixed point $v_a \in C_1$ for each $a \in S^0$.

\section{Proofs of technical results}
\label{appendix}
\begin{proof}[Proof of Lemma \ref{wcor2.3}]
Since $[{\rm supp}\ \mu]\cap S^0\not= \emptyset$, there exists $n\in \N^*$ such that $\mu^n$
is a barycenter with nonzero coefficients of $\mu_0\in M^1(S^0)$ and $\mu_1\in M^1(S)$:
$$
\mu^n = u\mu_0 + (1-u)\mu_1,\quad u\in(0,1).
$$
Then
$$
m^n = \int a\mu^n(da) = u\int a\mu_0(da) + (1-u) \int a \mu_1.
$$
If $x\in C$:
$$
m^n x = u\int ax\mu_0(da) + (1-u) \int ax \mu_1(da).
$$
Clearly $ax \in C^0$ if $a\in S^0$, hence by convexity of $C^0$ and $C$: $\int ax\mu_0(da)\in C^0$,
$\int ax\mu_1(da)\in C$. Since $u>0$, we get $m^nx\in C^0$. The existence and uniqueness  of the dominant
eigenvector $v\in C_1^0$ for $m^n$ follows. Then, since $m$ and $m^n$ commute, $v$ is also the unique dominant eigenvector
for $m$:
$$
\int av \mu(da) = r(m)v.
$$ Similar results are valid for $(m^*)^n$ and $m^*$: $m^*$ has a unique dominant eigenvector $v^*$
which belongs to the interior of $C_1^*$ and $m^* v^* = r(m)v^*$.
Since $v^*$ is in the interior of $C_1^*$, there exists a constant $D$
such that for any $a\in S$: $\norm{a}=\norm{a^*}\le D\norm{a^* v^*}$. The same argument as in
the proof of the  Lemma  gives that, if $x\in C_1^0$, there exists $D'>0$
such that for any $v'\in C_1^*$: $|v'|\le D'\is{v'}x$. It follows $\norm{a}\le DD'\is{a^*v^*}x$.
Hence for any $n\ge 1$,
$$\E\norm{A_n\ldots A_1}\le DD'\is{\E(A^*)^nv^*}{x} = DD' r^n(m)\is{v^*}x.
$$
In the limit $\kappa(1)\le r(m)$. On the other hand, for any $n\ge 1$,
$\norm{\E A^n}\le E\norm{A_n\ldots A_1}$, hence in the limit $r(m)\le \kappa(1)$,
and finally $r(m)=\kappa(1)$.

Considering the continuous function $v^*$ on $C$ defined by $v^*(x)=\is{v^*}x$,
we have $$
Pv^*(x) = \int\is{v^*}{ax}\mu(da) = \is{m^*v^*}{x} = r(m)\is{v^*}x,
$$ hence $P v^* = r(m) v^*$. The uniqueness in  Proposition \ref{wthm2.2} gives $e^1(x)= \frac{\is {v^*}{x}}{|v^*|_\8}$.
\end{proof}

\begin{lem}
\label{lem: mtgconvergence}
Assume that $C$ is a proper convex closed cone with nonempty interior such that the dual cone $C^*$ has also nonempty interior. If $Y_n$ is a $C_+$ valued martingale, then $Y_n$ converges a.s. to some $C$ valued random variable $Y$.
\end{lem}
\begin{proof}
It is sufficient to prove, purely geometrical observation, that there exists a basis $\{e_j\}_{j=1}^d$ of $V$ such that
\begin{equation}
\label{eq: pos basis}
C \subset \Big\{  \sum_{j=1}^d a_j e_j:\ a_j\ge 0\ \forall j   \Big\}.
\end{equation}
Since then, one can express the martingale $Y_n$ in terms of this basis. The coordinates form  positive martingales convergent a.s.

\medskip

Since $C^*$ has nonempty interior, $C_+$ must be contained in an open halfspace of $V$ and without any loss of generality we may assume that
\begin{equation}
\label{eq: cone pos}
C_+ \subset \big\{(x,x_d)\in \R^{d-1}\times \R:\; x_d>0   \big\}.
\end{equation}
Let us define the hyperplane $H = \{ x\in \R^d:\; x_d=1  \}$ and the set  $B = C\cap H$. Then $C = \R_+\times B = \{\lambda x:\; \lambda \ge 0, x\in B\}$.
We will prove that the set $B$ is compact.  As the intersection of two closed convex sets it must be closed and convex. $B$ is also bounded. Indeed, assume that $B$ is unbounded. Then by convexity of $B$, there exists an infinite ray inside $B$, i.e. there are $b\in B$ and nonzero $h\in \R^{d-1}\times\{0\}$ such that
$\{ b+ \lambda h:\; \lambda \ge 0 \}\subset B$. Therefore for any $\lambda\ge 0$, $\frac 1{\lambda}(b+\lambda h) = b/\lambda + h $ belongs to the cone $C$ and passing with parameter $\lambda$ to infinity, we deduce that $h\in C$, which contradicts to \eqref{eq: cone pos}. This proves compactness of $B$.

Now let us define a basis $\{e_j\}_{j=1}^d$ of $V$ as follows. Fix a large constant $N$. For $j\le d-1$, the vector $\{e_j\}$ contains 1 on the $j$th coordinate, and $0$'s  on all the other (thus these vectors, restricted to $\R^{d-1}$ form a canonical basis). Let $e_d = (-1,-1,..,-1, 1/N)$. We will prove that this chosen basis satisfies \eqref{eq: pos basis}.

Given a subset $A$ of $\{1,2,..,d-1\}$ define the vectors
$$
h_A = N\bigg( 2\sum_{j\in A} e_j +  e_d \bigg)
$$
Then all the vectors $h_A$ are of the form $(\pm N,\ldots,\pm N, 1)$, with $+N$ on the coordinates exactly from the set A. Let $B_N$ be the convex hull of the vectors $\{h_A\}$. We may choose large $N$  such that $B\subset B_N$. Finally we have
$$
C_+ = \R_+^* \times B \subset \R_+^* \times B_N \subset \Big\{  \sum_{j=1}^d a_j e_j:\ a_j\ge 0\ \forall j   \Big\},
$$ thus we obtain \eqref{eq: pos basis} and complete proof of the Lemma.
\end{proof}

\newcommand{\W}{{\mathcal W}}
\begin{lem}\label{lem:strongirreducible}
{ Let $\Gamma := [\supp\, \mu]$ satisfy $(\H)$ and assume that the restriction of $\Gamma$ to $V(\Gamma)$ consists of invertible linear operators. %and  that no proper subspace $W$ with $W \cap (C \cup C^*) \neq \{0\}$ is $\Gamma$-invariant.
Then $\Gamma$ is $C$-strongly irreducible, i.e. there is no finite union $\W=\bigcup_{i=1}^n W_i$ of proper subspaces $W_i \subsetneq V(\Gamma)$ with $W_i \cap C_+ \neq \emptyset$ for all $i$, satisfying $\Gamma \W \subset \W$.}
%An analogous result holds if the restriction of $\Gamma^*$ to $V(\Gamma^*)$ consists of invertible matrices, then $\Gamma^*$ is $C^*$-strongly irreducible.
\end{lem}

\begin{proof}
{ Assume that there is a finite union $\mathcal{W}=\bigcup_{i=1}^n W_i$ of proper subspaces $W_i \subsetneq V(\Gamma)$ with $W_i \cap C_+ \neq \emptyset$  which is $\Gamma$-invariant. Since $V(\Gamma)$ is the minimal invariant subspace, we have $n \ge 2$.
%By assumption, $n\ge 2$.
Then $\mathcal{W} \cap C_1$ consists of $n$ compact connected components $J_i$ with positive distance w.r.t. the metric $b(\cdot,\cdot)$. W.l.o.g. let $L_1, L_2$ have minimal distance $b_{\min}$.

Let $a \in \Gamma$. Observe that action of the restriction $\bar a$ to $V(\Gamma)$ preserves the dimension of subspaces of $V(\Gamma)$ due to invertibility, hence $\bar a \mathcal W \subset \mathcal W$ implies that $\bar a$ permutes the subspaces and thus the $J_i$ as well. Fix $x_1 \in J_1$, $x_2 \in J_2$ s.t. $b(x_1, x_2) = b_{\min}$. In particular, since $\bar a \in {\rm GL}(V(\Gamma))$ permutes the $J_i$, $b(a \cdot x_1, a \cdot x_2) \ge b_{\min} = b(x_1, x_2)$. Now let $\bar a \in \Gamma \cap S^0 \cap {\rm GL}(V(\Gamma))$ -- such $a$ exists by condition $(\H)$. By Lemma \ref{lem:metric} and the invertibility of $\bar a$, $$0 < b(a \cdot x_1, a \cdot x_2) \le d(a) b(x_1, x_2) < b(x_1, x_2),$$ which gives a contradiction. }
\end{proof}

\begin{proof}[Proof of Lemma \ref{inverti}]
{
Let $\nu$ be the law of $Z$ a non-trivial fixed point of $T$, supported on $V(\Gamma)\cap C_+$. Hence, in order to prove the lemma, we can assume $V=V(\Gamma)$, $C = V(\Gamma)\cap C$, in particular $\Gamma$ consists of invertible operators.
 We say that a subspace $W\subset V$ is $C$-positive if its orthogonal $W^{\perp} = \{ v'\in V; v'=0 \mbox{ on } W \}$ is generated by elements of the cone $C^*$. Assume that for some $u\in C^*$, $r\ge 0$, $\P[\langle  Z,u \rangle=r ]>0$, hence the hyperplane $\{\is xu = 0\}$ is $C$-positive.
We observe that { $V(\Gamma^*)^{\perp} = V^-(\Gamma)\subsetneq V$}. For an affine subspace $W$ we denote by $\overline W\subset V$ its direction.
Therefore, we are going to show that $\nu(W)=0$ for any $C$-positive proper affine subspace $W \subsetneq V$ with $W \cap C \neq \emptyset$.
\medskip

{\sc Step 1}. We consider an affine recursion associated with $(T,\nu)$. % and we use the same argument as in \cite{BDG}, Lemma 2.5, with a few complements.
 Let $Z,Z_i$ $(1\le i \le N)$ be i.i.d random variables with law $\nu$, $B = \sum_{i=2}^N A_i Z_i$. Let us write the fixed point equation \eqref{iteration} as $Z=_d A_1 Z_1+B$. Denote by $\eta$ the law of $B$ and let us consider the probability measure $p=\mu\otimes \eta$ on the affine group $H = {\rm End(V)}\ltimes V$.

We first show that the action of the support of $p$ on $V$ has no fixed points. Otherwise, for some $x\in V$ and $p$ a.e., $(a,b)\in H$ we have $x = ax+b$. Hence $\mu$ a.e., for some fixed $y$, $ax=y$ and $b=x-y$. In other words $\eta$ is the unit Dirac mass at $x-y$. This gives that the law of $A_1 Z_1$ is the Dirac unit mass at $\frac{x-y}{N-1}$, hence $\nu = \delta_z$ with
{ $z= \frac{N}{N-1} (x-y) \neq 0$. This means that $z \in C_+$ is a joint eigenvector of the elements of $\supp\, \mu$ with a fixed eigenvalue $1/N$. This contradicts the aperiodicity of $[\supp\, \mu]$. }
% $({\rm supp }\ \mu)z = \frac{x-y}{N-1}$. By definition of $\nu$ we have $\mu^{\otimes N}$ a.e.: $z = (a_1+\cdots+a_N)z$, in particular for each $a\in {\rm supp}\ \mu$, $Naz = z$. Then, if $z\not=0$, the line $W$ generated by $z$ is a nontrivial ${\rm supp}\ \mu$ - invariant subspace with $W\cap C\not=\{ 0 \}$, which contradicts condition ($\H$), since $d> 1$. Since $\supp\ \nu \subset C_+$, $\nu = \delta_0$ is excluded.

\medskip

{\sc Step 2.}
Let $\W$ be the set of affine subspaces of $V$ with positive $\nu$-mass, $\dim W< \dim V$, $C$-positive direction and minimal dimension,  hence $\W$ is non void. If $W,W'\in \W$ and $W\not= W'$, then $\dim (W\cap W') < \dim W$, hence $\nu(W\cap W')=0$. Since $\sum_{W\in \W}\nu(W)\le 1$, it follows that $\sup\{ \nu(W); W\in \W \} = \nu(W_0)$ for some $W_0\in \W$ and the set $\W_0$ of such $W_0$'s is finite. From the fixed point equation, we have if $W_0\in \W_0$, $\nu(W_0) = \int(h\nu)(W_0)dp(h)$.

By assumption, $h$  is invertible on $V$ and thus we have $\dim h^{-1}(W_0) = \dim W_0$ or $\nu(h^{-1}(W_0)\cap C_+)=0$, hence $\nu(h^{-1}W_0) \le \nu(W_0)$. Then the equation above gives $h W_0\in \W_0$ $p$ - a.s., hence the finite set $\W_0$ is invariant under the action of the subgroup $H_0$ of $H$ generated by ${\rm supp}\ p$.

\medskip

{\sc Step 3}.  If ${\rm dim } W_0=0$, then $\W_0$ is a $H_0$-invariant finite set. Hence its barycenter is $H_0$-invariant, is a $\supp\, p$ - fixed point and this is a contradiction with the result of Step 1.

\medskip

{\sc Step 4}. Since ${\rm dim } W_0>0$, then for $\overline \W_0=\{\overline W_k\}$, we consider $\overline \W_0^\perp=\{\overline  W_k^\perp \}$ and, since the $\overline W_k$ are $C$-positive, each $\overline W_k$ is generated by elements of $C^*$. Since $0 < \dim \overline W_k = \dim V - \dim \overline W_k^\perp$, these subspaces are proper. We are going to show $\overline W_k \subset V^-(\Gamma)$. For $h=(a,b)\in H_0$, the condition $a \overline W_i = \overline W_j$ implies $(a^*)^{-1}\overline W_i^\perp = \overline W_j^\perp$, hence $\overline \W_0^\perp$ is $\Gamma^*$-invariant.  Then Lemma \ref{lem:invariant set}
gives $V(\Gamma^*)\subset \overline W_k^{\perp}$ for each $k$ and consequently $\overline W_k\subset V^-(\Gamma)$. We observe that, since $\Gamma$ preserves $V^-(\Gamma)$, the affine action of $H_0$ on $V/V^-(\Gamma)\not=\{0\}$ is well defined, and we can project equation \eqref{iteration} on $V/V^-(\Gamma)$. Also, since $V^-(\Gamma)$ is $C$ positive, $\Gamma$ preserves the proper convex cone, which is the projection of $C$ in $V/V^-(\Gamma)$. It follows that the dominant eigenvalues  of the projection of $\Gamma$ in ${\rm GL} (V/V^-(\Gamma))$ are the same than those of $\Gamma$ in $V$, hence the projection of $\Gamma$ is aperiodic. Since any $\overline W_k\in \overline \W_0$ is contained in $V^-(\Gamma)$, the projection of $ \W_0$ in $V/V^-(\Gamma)\not=\{0\}$ is a finite set, invariant under the action of $H_0$. As above we get a contradiction with the aperiodicity of the projection of $\Gamma$.

%Write $V = V(\Gamma^*) \oplus V(\Gamma^*)^\perp$, then each $a \in \Gamma^*$ factorizes over this decomposition, hence we may as well decompose each $W_k^\perp %=W_k^{(1)} \oplus W_k^{(2)}$ in a component along $V(\Gamma^*)$ and a perpendicular one. Then $\bigcup W_k^{(1)}$ is a $\Gamma^*$-invariant finite union of %subspaces of $V(\Gamma)$. By Lemma \ref{lem:strongirreducible}, then each $W_k^{{1}}$ either equals $V(\Gamma^*)$ or $\{0\}$. But the latter possibility is %excluded, since by Lemma \ref{lem:invariant set}, the intersection of each $W_k^\perp$ with $V(\Gamma^*)$ is non-trivial.%
	%
%We conclude that $V(\Gamma^*) \subset W_k^\perp$ for each $k$ and consequently, $W_k \subset V(\Gamma^*)^\perp$ for each $k$. Factoring out $V$ by %$V(\Gamma^*)^\perp$, we can project the equation \eqref{iteration} on $V / V(\Gamma^*)^\perp$. Then the projection of $\W_0$ is a finite $H_0$-invariant set. As %above we get a contradiction.	 
}

For the last assertion about $Y$, we just recall  that Theorem \ref{existence} gives $Y\in V(\Gamma)\cap C_+$. 
\end{proof}

\section{Direct Riemann integrability}
\label{app: renewal}

Now we are going to explain relations between the definition of directly Riemann integrable functions given in \cite{K2} and our definition \eqref{dri}.

\medskip

In \cite{K2} the definition is as follows.
We define a family of subsets of $C_1$
$$
D_k=\bigg\{ x\in C_1:\; \Q^{\chi}_x\Big[V_m \ge \frac mk, \forall m\ge k\Big] \ge \frac 12
\bigg\}.
$$ Of course $D_k$ is an increasing family. We put $D_0=\emptyset$.

We say that a  function $g:C_1\times \R \to \R$ is directly Riemann integrable (dRi) if
it is $\B(C_1)\times \B(\R)$ measurable and satisfies
\begin{equation}
\label{dri1}
\sum_{k=0}^\8 \sum_{l=-\8}^\8 (k+1) \sup \Big\{  |g(x,t)|:\; x\in D_{k+1}\setminus D_k,
l\le t\le l+1 \Big\}<\8
\end{equation}
and if for every fixed $x\in C_1$ and the function $t\mapsto g(x,t)$
is Riemann integrable on $[-L,L]$, for $0<L<\8$.

\begin{lem}
\label{4.11}
For any  $h\in C_b(C_1\times \R)$ condition \eqref{dri} implies \eqref{dri1}.
\end{lem}
\begin{proof}
 We will prove that if we take
sufficiently large $k$, then $D_k = C_1$ and thus  the sum over $k$ in \eqref{dri1} is indeed finite and \eqref{dri}  implies \eqref{dri1}.

\medskip

Take $\d=\frac 1{2 d_1}$, where $d_1$ is as in Lemma \ref{wlem2.5}.
Define a random variable $Z(x) = \inf_{n\in\N}\frac{|S_n x|}{|S_n|}$. In view of
Lemma \ref{wlem2.7}, $Z(x)$ is strictly positive, $\Q^{\chi}$ a.e.
Let $\{w_i\}_{i\in\N}$ be a dense countable subset of $C_1$. Then, for every $w_i$
there exists $\eps(w_i)>0$ such that
$$
\Q^{\chi}[Z(w_i)\le \eps(w_i)] \le \frac{\d}{2^{i+1}}.
$$ By compactness of $C_1$ there exists a finite subset $\{x_1,\ldots, x_K\}$
of the sequence $\{w_i\}$ such that the balls $B(x_i,\eps(x_i)/2)$ cover $C_1$
and moreover
$$
\Q^{\chi}\big[ Z(x_i)\le \eps(x_i), \mbox{ for some } i\le K
\big] < \frac \d 2.
$$
Since, by Theorem \ref{wthm2.4}, $\lim_{n\to\8}\frac 1n \log|S_n(\o)x_i|= \a>0$, $\Q^{\chi}$ a.s.
for every $i$, there exists $k$ such
that the set
$$
\O_1=\bigg\{
\o:\; \frac{\log|S_n(\o) x_i|}{n} > \frac 2k \mbox{ and } Z(x_i) > \eps(x_i) \mbox{ for } n\ge k, 1\le i\le K
\bigg\}
$$
satisfies
$$
\Q^{\chi}(\O_1)>1-\d.
 $$
 We will show that the number $k$ is exactly the index we are looking for.

 Now take  arbitrary $y\in C_1$ and $x_i$ such that $y\in B(x_i,\frac{\eps(x_i)}{2})$.
 Notice that
 $$
 \bigg|   \frac{|S_n(\o)y|}{|S_n(\o)x_i|} -1  \bigg| \le \frac{\eps(x_i)}2
 \frac{|S_n(\o)|}{|S_n(\o)x_i|}\le \frac 12 \quad \mbox{ for } \o\in \O_1, n\ge k.
 $$
 Notice $|\log x|< 2|x-1|$ for $x\in (\frac 12,\frac 32)$, hence
 $$
 \log    \frac{|S_n(\o)y|}{|S_n(\o)x_i|} \ge -1  \quad \mbox{ for } \o\in \O_1, n\ge k,
 $$
 that implies
 $$
 \frac {\log|S_n(\o)y|}n \ge   \frac {\log|S_n(\o)x_i|}n -\frac 1n > \frac 1k
  \quad \mbox{ for } \o\in \O_1, n\ge k.
 $$
 Therefore $\Q^{\chi}(\O_2)>1-\d$ for
 $$
 \O_2 = \Big\{\o:\;  \frac {\log|S_n(\o)y|}n > \frac 1k \mbox { for } n>k
  \Big\}
 $$
and finally, by Lemma \ref{wlem2.5}
$$
\Q_y^{\chi}(\O_2) = 1- \Q_y^{\chi}(\O_2^c) \ge 1- d_1 \Q^{\chi}(\O_2^c) \ge 1-d_1 \d = \frac 12,
$$
thus $y\in D_k$.
\end{proof}

 \newpage

\section{List of Symbols}
\begin{tabularx}{\textwidth}{ cX }
$[\cdot]_s$ & H\"older norm, $[f]_s = \sup_{x,y \in C_1} \frac{\abs{f(x)-f(y)}}{b(x,y)^{\bar s}}$  \\
$\abs{\cdot}$ & norm on $V = \R^d$ associated to $\is \cdot \cdot$ resp. generation of a vertex in the tree\\
$a \cdot x$ & $a \cdot x = \abs{ax}^{-1}ax$ \\

$A$ & random matrix with law $\mu$ \\
$A^i(\gamma)$ & $A^i=\1_{\{i \le N \}}A_i$, $((A^i(\gamma))_i)_\gamma$ i.i.d. copies of $(A^i)_i$, attached to the vertices $\gamma$ of the tree % $\bigcup_{n=0}^\infty \N^n$.
\\

$b(x,y)$ & metric on $C_1$, see Lemma \ref{lem:metric} \\
$\b_u$ & law of $\sum_{i=2}^N \is{A_iX_i}{u}$\\

$C$ & proper closed convex cone in $V=\R^d$ with nonempty interior \\
$C^*$ & dual cone $C^*=\{x \in V \, : \, \is x y \ge 0 \text{ for any } y \in C \}$ \\
$C_+, C_1$ & $C_+=C \setminus \{0\}$, $C_1=\{x \in C \, ; \, \abs{x}=1 \}$ \\

$d(a)$ & Lipschitz constant of $a \in S$ w.r.t. the metric $b$ on $C_1$; $d(a) < 1$ iff $a \in S^0$.\\
$\Delta(\Gamma) $ & $\Delta(\Gamma)=\{ \lambda_a \, : \, a \in \Gamma \in S^0\}$ \\

$\E_x^s$ & expectation symbol of $\Qxs$ \\
$\es, \est$ & eigenfunctions $\Ps \es = \kappa(s) \es$, $\Pst \est = \kappa(s) \es$, strictly positive and $s$-H\"older \\
$\wt \es$ & $\wt \es(x) = \int \is{x}{y}^s \, \nust(dy)$, proportional to $\es$.\\

$\F_n$ & $\F_n = \sigma\Bigl( \bigl(A^i(\gamma)\bigr)_{i \in \N} \, : \, \abs{\gamma} < n \Bigr)$ \\
$f(u,r) $ & $f(u,r)=\int_r^\8 s\, \P(\is{X}{u} \in ds)$ \\% = r\P(\is X u > r) + \int_r^\8 P(X > t) dt$\\

$\Gamma$ & $\Gamma = [\supp\,\mu]$ the semigroup generated by $\supp \mu$. \\

$(\H)$ & (a) each $a \in \Gamma$ satisfies $a C^0 \subset C^0$ and $a^* (C^*)^0 \subset (C^*)^0$; (b) $\Gamma \cap S^0 \neq \emptyset$\\

$\iota(a)$ & $\iota(a) = \inf\{\abs{ax} \, : \, x \in C_1 \}$\\
$I_\mu$ & $I_\mu = \{ s \ge 0 \, : \, \Erw{\norm{A}^s} < \8\}$\\

$\kappa(s)$ & $\kappa(s):= \lim_{n \to \infty} \Erw{\norm{A_n \cdots A_1}^s}^{1/n}$, spectral radius of $\Ps$ and $\Pst$. \\

$L(\gamma)$ & $L(\emptyset)={\rm Id}$, the identity matrix, $L(\gamma i) = L(\gamma) A^i(\gamma)$\\
$\Lambda(\Gamma)$  & the closure of $\{ v_a \, : \, a \in \Gamma \cap S^0 \}$ \\
$\lambda_a$ & dominant eigenvalue of the matrix $a \in S^0$ (Perron-Frobenius eigenvalue)\\

$m$ & $m=\Erw{A}$ \\
$\mu$ & law of $A$, prob. measure on $S \subset End(V)$. \\

$N$ & random number of summands in the smoothing transform $T$\\
$\nus, \nust$ & eigenmeasures $\Ps \nus = \kappa(s) \nus$, $\Pst \nust = \kappa(s) \nust$, supported on $\Lambda([\supp\, \mu])$, resp. $\Lambda([(\supp\, \mu)^*])$\\

$\Omega$ & $\Omega=S^\N$ with shift $\theta$.\\

$\Ps$ & $\Ps \psi(x) = \int_S \abs{ax}^s \psi(a \cdot x) \,\mu(da)$ \\
$\Pst$ & $\Ps \psi(x) = \int_S \abs{a^*x}^s \psi(a^* \cdot x) \,\mu(da)$ \\
$\pi^s$ & $\pi^s(dx) = (\nus(\es)) \es(x) \nus(dx)$, invariant prob. measure of $Q^s$.\\

$q_n^s(x,a)$ & $q_n^s(x,a) = \frac{\abs{ax}^s}{\kappa(s)^n} \frac{\es(a \cdot x)}{\es(x)}$\\
$\Qxs$ & projective limit of the system $q_n^s(x,\cdot) \mu^{\otimes n}$ on $\Omega=S^{\N}$\\
$\Q^s$ & $\int \Qxs\, \pi^s(dx)$\\
$Q^s$ & Markov operator on $C_1$, $Q^s \phi(x) = (\kappa(s) \es(x))^{-1} \Ps(\phi \es)(x)$ \\

$\bar s$ & $\bar s = \min\{s,1\}$\\
$S$ & $S= \{ a \in {\rm End}(V) \, ; \, aC_+ \subset C_+, \, a^*C_+^* \subset C_+^* \}$\\
$S^0$ & $S^0=\{a \in S \, : \, aC_+ \subset C^0 \}$\\
$S_n$ & $S_n = A_n \ldots A_1$\\

$T$ & smoothing transform, maps a law $\rho$ to the law of $\sum_{i=1}^N A_i Z_i$ where $(A_i), (Z_i), N$ are independent, $(A_i)$ are i.i.d. with law $\mu$ and $(Z_i)$ are i.i.d. with law $\rho$\\
$\tau(x)$ & $\tau(x) = \inf\{ \norm{a}^{-1} \abs{ax} \, : \, a \in S\}$; strictly positive for $x \in C^0$. \\

$v_a$ & dominant eigenvector of the matrix $a \in S^0$ (Perron-Frobenius eigenvector) \\
$v, v^*$ & dominant eigenvectors of $m= \Erw{A}$ resp. $m^*= \Erw{A^*}$ \\

$Y_n$ & Mandelbrot's cascade $Y_n = \sum_{\abs{\gamma}=n} L(\gamma) v$, with $\Erw{N}\, \Erw{A}\, v = v$.
\end{tabularx}

\bibliographystyle{alpha}
\newcommand{\etalchar}[1]{$^{#1}$}

\end{document}